\documentclass[11pt]{amsart}
\usepackage{amsthm,amsmath,amsxtra,amscd,amssymb,xypic,hyperref,cleveref,xparse,xspace,mathtools}
\usepackage[all]{xy}
\usepackage{tikz-cd}
\numberwithin{equation}{section}
\usepackage{bm}
\usepackage{caption,subcaption}

\theoremstyle{plain}
\newtheorem{thm}{Theorem}[section]
\newtheorem{lm}[thm]{Lemma}
\newtheorem{prop}[thm]{Proposition}
\newtheorem{cor}[thm]{Corollary}

\theoremstyle{definition}
\newtheorem{dfx}[thm]{Definition}
\newenvironment{df}
{\pushQED{\qed}\dfx}
{\popQED\enddfx}
\newtheorem{remx}[thm]{Remark}
\newenvironment{rem}
{\pushQED{\qed}\remx}
{\popQED\endremx}
\newtheorem{exax}[thm]{Example}
\newenvironment{exa}
{\pushQED{\qed}\exax}
{\popQED\endexax}

\AtBeginDocument{%
   \def\MR#1{}
 }

\def\={\;=\;}  \def\+{\,+\,}

\newcommand{\inj}{\hookrightarrow}
\newcommand{\surj}{\twoheadrightarrow}

\newcommand{\codim}{\operatorname{codim}}
\newcommand{\GL}{\operatorname{GL}}

\newcommand{\ord}{\operatorname{ord}}

\newcommand{\mult}{\operatorname{mult}}
\newcommand{\Hur}{\operatorname{Hur}}
\newcommand{\GM}{\operatorname{GM}}

\newcommand{\Span}{\operatorname{span}}

\newcommand{\divisor}[1]{{\rm div }\left( #1 \right)}

\newcommand{\BigStar}{\mbox{\normalfont\Large\bfseries  $\boldsymbol\ast$}}
\newcommand{\BigZero}{\mbox{\normalfont\Large\bfseries  0}}

\newcommand{\Teichmuller}{Teich\-m\"uller\xspace}

\newcommand\thc{top-horizontal-crossing\xspace}
\newcommand\nhc{non-top-horizontal-crossing\xspace}

\newcommand{\hcc}{\thc cycle\xspace}
\newcommand{\hccs}{\thc cycles\xspace}

\newcommand{\nhccs}{non-\thc cycles\xspace}
\newcommand{\hces}{\thc equations\xspace}
\newcommand{\nhces}{non-\thc equations\xspace}

\newcommand{\nhce}{non-\thc equation\xspace}

\newcommand{\calC}{\mathcal C}

\newcommand{\calM}{{\mathcal M}}

\newcommand{\CC}{{\mathbb{C}}}

\newcommand{\PP}{{\mathbb{P}}}

\newcommand{\RR}{{\mathbb{R}}}
\newcommand{\ZZ}{{\mathbb{Z}}}

\newcommand{\modulin}[1][g,n]{{\mathcal M}_{#1}}
\newcommand{\omodulin}[1][g,n]{{\Omega\mathcal M}_{#1}}

\newcommand{\poles}{\underline{p}}
\newcommand{\zeroes}{\underline{z}}

\newcommand{\lG}{{\Gamma}}     
\newcommand{\ltop}[1][e]{\ell(#1^{+})}                   
\newcommand{\lbot}[1][e]{\ell(#1^{-})}                  
\newcommand{\Mrel}{$M$-cross-related\xspace}
\newcommand{\rrefrel}{rref-cross-related\xspace}
\newcommand{\Mequiv}{$M$-cross-equivalence\xspace}
\newcommand{\rrefequiv}{rref-cross-equivalence\xspace}

\newcommand{\hor}{{\mathrm {hor}}}
\newcommand{\ver}{{\mathrm {ver}}}

\DeclareDocumentCommand{\Ehor}{O{\lG}}{E^{\hor}(#1)}   
\DeclareDocumentCommand{\Ever}{O{\lG}}{E^{\ver}(#1)}  
\DeclareDocumentCommand{\Ehori}{O{i} O{\lG} }{E^{\hor}_{(#1)}(#2)}   

\DeclareDocumentCommand{\abEhor}{O{\lG}}{\widetilde{E^{\hor}}(#1)}

\renewcommand{\Im}{{\operatorname{Im}}}
\newcommand{\Ann}{{\operatorname{Ann}\,}}

\DeclareDocumentCommand{\kmoduli}{ O{\mu} O{g} O{n} }{\calM_{#2,#3}(#1)}
\DeclareDocumentCommand{\barkmoduli}{ O{\mu} O{g} O{n} }{\overline{\calM}_{#2,#3}(#1)}
\DeclareDocumentCommand{\obarkmoduli}{ O{\mu} O{g} O{n} }{\Omega\overline{\calM}_{#2,#3}(#1)}

\DeclareDocumentCommand{\LMS}{ O{\mu} O{g,n}} {{\Xi\overline{\calM}_{#2}(#1)}}
\DeclareDocumentCommand{\RBLMS}{ O{\mu} O{g,n}}{\Xi\widehat{\calM}_{#2}(#1)}

\newcommand{\oM}{{\overline M}}
\newcommand{\pM}{\partial M}
\newcommand{\DG}[1][\lG]{D_{#1}}
\newcommand{\pMG}[1][\lG]{\pM_{#1}}
\newcommand{\topl}{{\top}}
\newcommand{\toplr}[1][\beta]{ [{#1}_{\top}]}

\DeclareDocumentCommand{\relhomZ}{O{X} O{\zeroes} O{\poles}}{H_1(#1\setminus #3,#2;\ZZ)}
\DeclareDocumentCommand{\relhomC}{O{X} O{\zeroes} O{\poles}}{H_1(#1\setminus #3,#2;\CC)}
\newcommand\hlgr[1][i]{H_1(X_{(#1)}\setminus (\poles\cup \Lambda_{(#1)}^{hor}), \zeroes\cup \Lambda_{(#1)}^{ver,+};\ZZ)}

\newcommand\adC{\gamma}
\newcommand\crC[1][(i)]{\delta^{#1}}
\newcommand\hfC[1][(i)]{\alpha^{#1}}
\newcommand\lgad{K}
\newcommand\lgcr[1][i]{c(#1)}
\newcommand\lghf[1][i]{d(#1)}
\newcommand\van[1][e]{\lambda_{#1}}

\newcommand{\scl}[1][i]{t_{\lceil {#1}\rceil}}

\newcommand\logP[1][\gamma]{\Psi_{#1}}

\newcommand \lrra[1]{\left \langle #1\right \rangle}
\newcommand\rref{reduced row echelon form\xspace}

\usepackage{color}
\begin{document}
\title{Equations of linear subvarieties of strata of differentials}
\author{Frederik Benirschke}
\address{Mathematics Department, Stony Brook University, Stony Brook, NY 11794-3651, USA}
\email{Frederik.Benirschke@stonybrook.edu}
\author{Benjamin Dozier}
\email{Benjamin.Dozier@stonybrook.edu}
\author{Samuel Grushevsky}
\email{sam@math.stonybrook.edu}
\thanks{Research of the third author is supported in part by the National Science Foundation under the grant DMS-18-02116.}

\begin{abstract}
    We investigate the closure~$\oM$ of a linear subvariety~$M$  of a stratum of meromorphic differentials in the multi-scale compactification constructed in~\cite{BCGGMmsds}. Given the existence of a boundary point of~$M$ of a given combinatorial type, we deduce that certain periods of the differential are pairwise proportional on~$M$, and deduce further explicit linear defining relations. These restrictions on linear defining equations of~$M$ allow us to rewrite them as explicit analytic equations in plumbing coordinates near the boundary, which turn out to be binomial. This in particular shows that locally near the boundary $\oM$ is a toric variety, and allows us to prove existence of certain smoothings of boundary points and to construct a smooth compactification of the Hurwitz space of covers of $\PP^1$. As applications of our techniques, we give a fundamentally new proof of a generalization of the cylinder deformation theorem of Wright~\cite{wright} to the case of real linear subvarieties of meromorphic strata.
\end{abstract}
\maketitle

\section{Introduction}\label{sec:intro}
For an $n$-tuple of integers $\mu=(m_1,\dots,m_n)$ satisfying $m_1+\dots+m_n=2g-2$, $\omodulin(\mu)$ denotes the stratum of meromorphic differentials of type $\mu$, that is the locus of triples $(X,\underline{x},\omega)$ where $\underline{x}=\lbrace x_1,\dots,x_n\rbrace$ is a collection of distinct marked points on a smooth genus $g$ Riemann surface~$X$, and $\omega$ is a non-zero meromorphic differential on~$X$ such that its divisor of poles and zeroes is $\divisor{\omega}=\sum m_ix_i$. We call points $(X,\underline{x},\omega)\in \omodulin(\mu)$ {\em flat surfaces of type~$\mu$}, and will also write $\zeroes\subset X$ for the set of all zeroes of $\omega$, that is all $x_i$ such that $m_i\ge 0$, and $\poles\subset X$ for the set of all poles of~$\omega$. A natural set of local coordinates on the stratum is given by period coordinates: the integrals of $\omega$ over a chosen basis of $\relhomZ$.  The group $\GL^+(2,\RR)$ naturally acts on the stratum, linearly in period coordinates.  This dynamical system is the central object of study in \Teichmuller dynamics.

The foundational results of Eskin-Mirzakhani~\cite{eskinmirzakhani} and Eskin-Mirzakhani-Mohammadi~\cite{esmimo} show that for holomorphic strata (that is, if all $m_i$ are positive), the closure, in the Euclidean topology, of any $\GL^+(2,\RR)$ orbit is cut out by linear equations with real coefficients.  Throughout this paper, when we say ``linear equation'', we mean an equation with no constant term.  Filip~\cite{filip} showed that such orbit closures, commonly called {\em affine invariant manifolds}, are algebraic varieties.

In this paper we study a more general class of subvarieties of the strata than affine invariant manifolds. An {\em algebraic} subvariety~$M\subseteq\omodulin(\mu)$ of complex codimension~$m$ is called a {\em linear (sub)variety} if at any point it is locally
a finite union of linear subspaces, in period coordinates. Most of our results are in this generality, i.e.~allowing complex coefficients and allowing meromorphic strata --- note in particular that a linear equation with complex coefficients does not need to be preserved by the $\GL^+(2,\RR)$ action.
While we require~$M$ to be algebraic, note that recent examples of Bakker-Mullane~\cite{bamu} indicate that a locus
given locally in a meromorphic stratum by complex-linear equations may not be algebraic. Our paper continues in the spirit of the works of the second author~\cite{ben} and the first author~\cite{fred}, using and developing degeneration techniques for flat surfaces to prove various properties of linear subvarieties. In this work, we also obtain information about the geometry of defining equations. This could be used to understand (or rule out the existence of certain) linear subvarieties in general, while our more precise results for affine invariant manifolds could provide tools for classifying $\GL^+(2,\RR)$-orbit closures.

\smallskip
In~\cite{BCGGMmsds} the moduli space of multi-scale differentials $\LMS$ was constructed, such that $\omodulin(\mu)\subset\LMS$ is open dense, and such that the quotient $\PP\LMS=\LMS/\CC^*$, where $\CC^*$ scales the differential, is compact.  A key property of both $\LMS$ and $\PP\LMS$ is that they are \emph{smooth} (as complex orbifolds) algebraic varieties, with normal crossing boundary. A multi-scale differential is a stable Riemann surface $X$ together with a map $\ell:V(\Gamma)\surj \{0,-1,\dots,-L(\Gamma)\}$ from the set of vertices of the dual graph $\Gamma$ of $X$, and together with a collection $\eta$ of meromorphic differentials $\eta_v$ on the irreducible components $X_v$ of $X$ satisfying certain compatibility conditions (additionally one needs an enhancement of the level graph and a prong-matching, see below). The boundary $\partial\LMS$ is stratified, with {\em open} strata $\DG$ indexed by enhanced level graphs~$\lG$. Up to finite covers, each such stratum is a subspace of a product of certain strata of differentials given by certain residue conditions, and as such, it (more precisely, some cover of it, see~\Cref{rem:genstratadivisorial}) also admits local period coordinates, see~\cite[Sec.~4]{euler} for much more on the geometry of the strata, which we will also use below.

In~\cite{fred} the first author used a detailed analysis of the degeneration behavior of period coordinates to prove that the intersection $\pMG:=\pM\cap\DG $ of the boundary $\pM:=\oM\cap\partial\LMS$ of the closure of any linear variety~$M$ with any boundary stratum is locally given by linear equations in period coordinates on that boundary stratum. In this paper we investigate geometric properties and stratifications of boundaries of linear varieties. While our results are described and obtained locally near the boundary of $\LMS$, they provide global geometric information: given the existence of some degeneration, they restrict the defining equations of~$M$.

\subsection*{Defining equations of linear subvarieties}
We fix once and for all a boundary point $p_0\in\DG$, and work throughout in a small neighborhood~$U$ of~$p_0$ in~$\LMS$, that we may need to shrink further finitely many times. Recall that the edges $e\in E(\Gamma)$ of the dual graph are called horizontal or vertical depending on whether they connect vertices of same or different levels for $\ell$; we write $E(\Gamma)=:\Ehor\sqcup\Ever$. For a vertical edge $e\in\Ever$ we denote $\ell(e^\pm)$ the levels of its top and bottom vertices, respectively.

We will consider the defining equations of~$M$ at a point $p=(X,\omega)\in  M\cap U$. In general,~$M$ is an immersed, and not embedded, submanifold of the stratum, and we always require~$p$ to be a point where~$M$ is locally embedded (and so~$M$ is smooth at $p$).  We then call a {\em defining equation (of~$M$ at~$p$)} a linear equation~$F$ satisfied locally on~$M$ near~$p$. We think of~$F$ as an equation $F(X,\omega)=\int_\beta \omega=0$ for some relative homology class $\beta\in\relhomC$. The vector space of such defining equations~$F$ of~$M$ at~$p$ is locally constant on~$M$ in a neighborhood of~$p$.

Recall that corresponding to any node $e\in E(\lG)$ there is the pinching curve $\Lambda_e\subset X$. Denote by $\lambda_e$ the homology class of $\Lambda_e$, called the {\em vanishing cycle}. We say that {\em $F$ crosses $e$} if the intersection number $\langle \beta,\lambda_e\rangle$ is non-zero. We denote by $\Ehor[F]\subseteq\Ehor$ the set of all horizontal nodes crossed by~$F$.

We now define the notion of two nodes $e_1,e_2$ being {\em \Mrel}.
A particular case of being \Mrel is when there exists a defining equation~$F$ of~$M$ at~$p$ crossing the two nodes, i.e. $e_1,e_2\in \Ehor[F]$, while there is no defining equation $F'$ crossing a non-empty proper subset of the nodes of $F$, i.e. $\Ehor[F'] \subsetneq \Ehor[F]$, $\Ehor[F']\ne \emptyset$.  We define {\em \Mequiv classes} to be the equivalence classes in~$\Ehor$ generated by this particular case, and call two nodes \Mrel if they are in the same equivalence class  (see~\Cref{df:Mrelated} for a more detailed discussion).

Our first result is that the periods over the vanishing cycles for any two \Mrel nodes are proportional.
\begin{thm}[Periods over horizontal vanishing cycles are proportional]\label{thm:Mrltd}
For any pair $e_1, e_2\in \Ehor$ of \Mrel horizontal nodes, the integrals of $\omega$ over the corresponding vanishing cycles $\lambda_{e_1},\lambda_{e_2}$ are proportional on~$M$. In particular, the nodes are at the same level.
\end{thm}
In the case when~$M$ is an affine invariant manifold, the above can be deduced from Wright's Cylinder Deformation Theorem~\cite{wright}.  We give a fundamentally new proof of a generalization of the Cylinder Deformation Theorem (see \Cref{thm:cyldeformation} below), and the above is one of the key tools we use. See \Cref{exa:cross} for a simple example of a set of \Mrel nodes and how the above Theorem applies.

For periods over the vanishing cycles for vertical nodes, we have
\begin{thm}[Relations among periods over vertical vanishing cycles]\label{thm:vertres}
For any defining equation $F$ of~$M$ at $p$, and for any level $i\leq \topl(F)$ let $e_1,\dots,e_k\in\Ever$ be all the vertical nodes crossed by~$F$ such that $\ltop[e_j]>i\ge\lbot[e_j]$, i.e. all vertical nodes that cross the level transition between levels $i+1$ and $i$. Then the set of periods of~$\omega$ over vanishing cycles $\van[e_j]$ satisfy a linear relation on~$M$ near~$p$.
\end{thm}
Here recall that cutting~$X$ along $\Lambda_e$ for all $e\in\Ever$ decomposes $X$ into the level subsurfaces~$X=\cup_i X_{(i)}$. For a collection of paths $\beta'\subset X$ we call its top level $\topl(\beta')$ the maximal $i$ such that $\beta'\cap X_{(i)}\ne\emptyset$. The top level $\topl(F)$ denotes then the top level of the homology class $[\beta]$, which is defined to be the minimum of $\topl(\beta')$ over all collections of paths $\beta'$ representing the class $[\beta]$. In~\Cref{prop:eqsquantitative} we will prove a more precise version of this Theorem that gives the coefficients of such a relation.

\begin{exa}
Consider the flat surface in the stratum $\Omega\mathcal{M}_{3,3}(1,1,2)$ shown in \Cref{fig:residue_reln}.  We claim that \Cref{thm:vertres} implies that there is no linear subvariety~$M$ containing this surface that is locally described by the single equation
$$
F= \int_{\gamma_1} \omega - \int_{\gamma_2} \omega = 0.
$$

\begin{figure}[h]
  \begin{center}
    \begin{subfigure}[b]{0.65\textwidth}
      \includegraphics[scale=0.59]{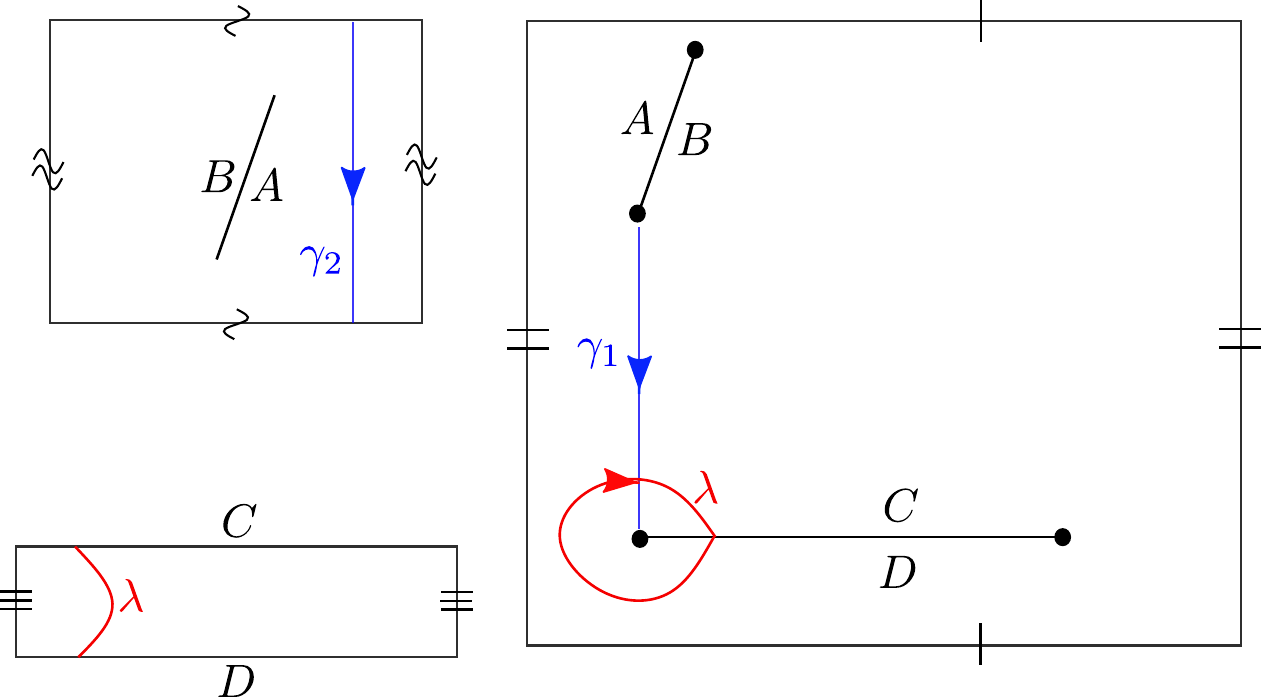}
      \caption{Flat surface in $\Omega\mathcal{M}_{3,3}(1,1,2)$. } 
      \label{fig:residue_reln}
    \end{subfigure}
    \begin{subfigure}[b]{0.34\textwidth}
      \includegraphics[scale=0.41]{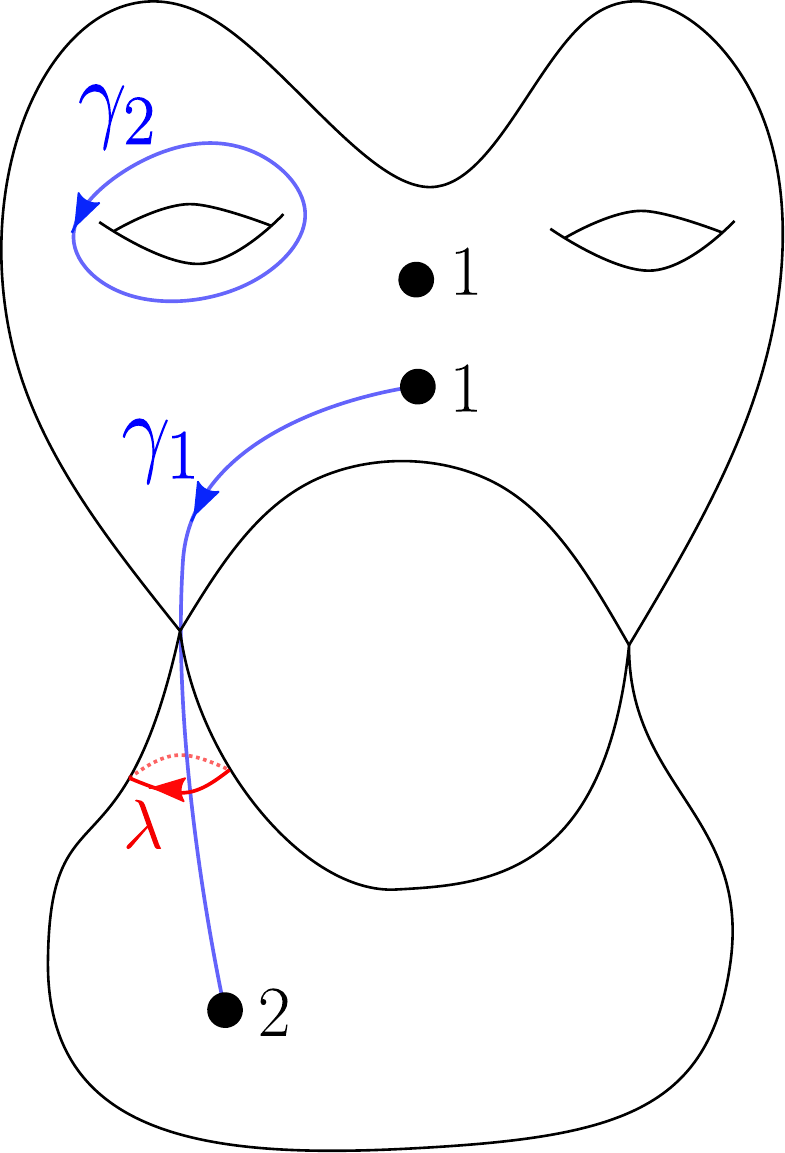}
      \caption{Limit as $\lambda\to 0$}
      \label{fig:residue_reln_limit}
    \end{subfigure}
    \caption{}
  \end{center}
\end{figure}

In fact, if there were, we could degenerate, staying in~$M$, by sending the period of the side labeled $\lambda$ to $0$ (and keeping the rest of the surface unchanged).  This gives a family of surfaces converging to a point $p_0\in \Xi \overline{\mathcal{M}}_{3,3}(1,1,2)$ with two levels, no horizontal nodes, and the class $\lambda$ is the vanishing cycle of a vertical node, with the curve~$\Lambda$ representing the homology class~$\lambda$ depicted in \Cref{fig:residue_reln_limit}. The period over this vanishing cycle is non-zero on all flat surfaces near $p_0$.  The equation $F$ crosses the (vertical) vanishing cycle $\lambda$, and no other vanishing cycles.  By \Cref{thm:vertres}, there is a non-trivial linear relation among the periods of the vanishing cycles crossed by $F$ on flat surfaces near $p_0$. Since this set of vanishing cycles is just $\{\lambda\}$, it follows that $\int_\lambda\omega = 0$ for all surfaces in~$M$ near $p_0$, which is impossible.

(In this particular case, such a linear manifold can also be ruled out using the Cylinder Deformation Theorem \cite{wright}, but in more complicated examples this would not be possible.)
\end{exa}

We further show that defining equations split into those that do not cross any horizontal nodes, and those that only cross horizontal nodes at their top level.
\begin{thm}[Decomposition of linear equations]\label{thm:decomp}
Any defining equation~$F$ of~$M$ at~$p$ can be written as a sum
\begin{equation}\label{eq:decomp}
F=H_1+\dots+H_k+G
\end{equation}
of defining equations of~$M$ at~$p$ (possibly with $k=0$) such that
\begin{enumerate}
\item each $H_j$ crosses a primitive collection of horizontal nodes of level $\topl(H_j)$, and no other horizontal nodes.
\item $\Ehor[H_j]\subseteq \Ehor[F]$ for any~$j$.
\item $G$ does not cross any horizontal nodes: $\Ehor[G]=\emptyset$.
\end{enumerate}
\end{thm}
Here primitive means that there does not exist a defining equation~$H'$ of~$M$ at~$p$ such that $\emptyset\neq\Ehor[H']\subsetneq\Ehor[H_j]$. This theorem gives a restriction for the form of the defining equations, given the existence of a boundary point of~$M$ with enhanced level graph~$\Gamma$.

Our methods allow us to control the dimensions of $\pMG$ and in particular describe the boundary strata that may contain irreducible components of~$\pM$.  Recall that the codimension $\codim_{\LMS}\DG$ of a boundary stratum is equal to $H(\lG)+L(\lG)$, where $H$ is the number of horizontal edges, and $L$ is the number of levels below zero.

\begin{thm}[Boundary components of~$M$]\label{thm:divisorial}
The general point of any irreducible component of the boundary $\pM$ is  contained in an open boundary stratum $\DG$ such that either $L(\lG)=1$  and $H(\lG)=0$, or such that $L(\lG)=0$.

In the latter case, for any pair of nodes $e_1,e_2\in E(\lG)$, there exist a defining equation $F$ of~$M$ such that $\Ehor[F]=\{e_1,e_2\}$.
\end{thm}
We note that in the latter case it follows that $e_1$ and $e_2$ are \Mrel and thus by~\Cref{thm:Mrltd} the periods over the two corresponding vanishing cycles are proportional on~$M$. What the Theorem shows is that for divisorial degenerations there moreover exists a defining equation that only crosses this pair of nodes (see~\Cref{thm:AIM} also for the related results for affine invariant manifolds in the minimal stratum).

Enumerating strata as above that could contain irreducible components of~$\pM$ is easy, and one can envision applying this to rule out the existence of certain linear subvarieties via degeneration analysis. Our most precise technical result in studying the equations and stratification of linear subvarieties is~\Cref{prop:eqsquantitative}, which gives the coefficients of defining equations, starting from the basis for defining equations of~$M$ at~$p$ taken in \rref with respect to a suitable homology basis.

Our proof will in fact yield a more general statement than the theorem above: for any linear subvariety of any boundary stratum $\DG$, the general points of its irreducible boundary components are obtained in the strata $\DG[\Gamma']$ where $\Gamma'$ is a purely horizontal or purely divisorial degeneration of $\Gamma$, obtained from $\Gamma$ either by introducing one new vertical level, or by introducing a new collection of cross-related horizontal edges. This will allow us to recursively apply this theorem and thus navigate the boundary stratification of a linear subvariety.

\subsection*{The analytic structure near the boundary of linear subvarieties}
Our next set of results provide some more detailed information about the geometry of a linear subvariety near its boundary. We recall that~$M$ is an immersed subvariety of the stratum, not an embedded one, and at its singular points we can only say that its local irreducible components are given by linear equations (this is simply to say that locally~$M$ looks like a finite union of linear subspaces). Similarly, when working with the closure $\oM$ near its boundary point $p_0\in\pM$, we will work separately with the local irreducible components~$\overline{Z}$ of $\oM$ at $p_0$ and denote $Z=\overline{Z}\cap M$.

Period coordinates on a stratum do not extend to the boundary, instead we have analytic {\em plumbing coordinates} in a neighborhood of the boundary. Using the precise information on the coefficients of defining equations for linear subvarieties obtained in \Cref{prop:eqsquantitative}, we can explicitly convert the linear equations in period coordinates into holomorphic equations in plumbing coordinates.
\begin{thm}\label{thm:localequations}
Let~$M$ be a linear subvariety and let $p_0\in\pM$. The local analytic equations for a local irreducible component $\overline{Z}$ of $\oM$ near $p_0$ can be computed explicitly from the defining equations of~$M$ at a smooth point of $Z$.  Analytically locally, $\overline{Z}$ is isomorphic to the product of $\CC^n$ and varieties defined by binomial equations. In particular, $Z$ is locally isomorphic to a (not necessarily normal) toric variety (see ~\eqref{eq:convert} and~\eqref{eq:binomial}).
\end{thm}
The strength of this result is that it allows us to describe the local structure of $\overline{Z}$ and of $\partial Z$ near $p_0$ very precisely.
Recall that an open stratum $\DG\cap U$ is contained in the closure of the open stratum $\DG[\lG']\cap U$ if and only if $\lG'$ is an undegeneration of~$\lG$. Any undegeneration is a composition of a {\em horizontal undegeneration}, which contracts some collection of horizontal edges of~$\lG$, and a {\em vertical undegeneration}, which contracts a number of level transitions in~$\lG$. The local defining equations allow us to show that certain such undegenerations occur in~$M$.
\begin{thm}\label{thm:transverse}
If $\pMG$ is non-empty, i.e.~if $p_0\in\pM\cap\DG$, then for any~$\lG'$ obtained from~$\lG$ by a composition of a vertical undegeneration, and of a horizontal undegeneration that smoothes some collection of \Mequiv classes, the intersection $\pMG[\lG']=\pM\cap\DG'$ is also non-empty.
\end{thm}
In~\cite{euler} it is shown that the union $D^\ver$ of all open boundary strata $\DG$ of $\LMS$ such that $\lG$ does not have horizontal nodes has simple normal crossings in $\LMS$. The above theorem says that in particular $\oM$ is generically transverse to $D^\ver$; a more precise statement is the following Corollary of~\Cref{thm:localequations}.
\begin{cor}\label{cor:eqsm}
For any local irreducible component $\overline{Z}$ of $\oM$ at $p_0$, if none of the defining equations of $Z$ cross any horizontal nodes, then $\overline{Z}$ is smooth and $\partial Z\subset\overline{Z}$ is normal crossing.
\end{cor}
All periods of exact differentials over vanishing cycles are equal to $0$, and thus horizontal nodes cannot arise when these are degenerated. This provides a very interesting situation where all defining equations are non-horizontal. The case of the double ramification locus will be treated by the first author in~\cite{fredDR}. In this paper we consider the Hurwitz spaces of covers: the spaces of branched covers $f:X\to\PP^1$ with prescribed branching over a number of points (see~\eqref{eq:Hurwdefine} for a precise definition). In~\Cref{prop:HW} we use exact differentials to realize Hurwitz spaces as linear subvarieties of the strata and to construct a {\em smooth} compactification of Hurwitz spaces.

\subsection*{Cylinder Deformation Theorem}
Given a boundary point $p_0\in\pMG$ and a node $e\in E(\lG)$, it is impossible to choose a relative homology cycle crossing $\van[e]$ that varies continuously in~$U$, as there is non-trivial monodromy. From the point of view of flat geometry, however, for any flat surface near $p_0$ one can naturally choose a ``long cylinder" around the vanishing cycle $\lambda_e$. Wright has proven the fundamental Cylinder Deformation Theorem, describing geometrically the types of deformations that can appear in affine invariant manifolds.  We give a fundamentally new proof in a somewhat more general context.

Recall that two cylinders on a flat surface are called {\em parallel} if the periods of their circumference curves are real multiples of each other. For an affine invariant manifold $M\subseteq \omodulin(\mu)$, cylinders $C_1,C_2\subset X\in\calM$ are called {\em $M$-parallel} if they are parallel on $X$, and parallel for all flat surfaces in a neighborhood of~$X$ in~$M$. An equivalence class of $M$-parallel cylinders on~$X$ is a maximal collection $\calC=C_1,\dots,C_d$ of cylinders on $X$ that are pairwise $M$-parallel. The original intuition for the Cylinder Deformation Theorem arose from the idea that affine invariant manifolds should be algebraic subvarieties (unproven at the time), which restricts the type of linear equations that are possible.  However, the original proof used quite different methods, relying on deep results of Minsky-Weiss \cite{miwe} and Smillie-Weiss \cite{smwe} on the dynamics of the horocycle flow.  Our proof follows the strategy of the original intuition.

\begin{thm}[{Cylinder Deformation Theorem}] \label{thm:cyldeformation}
Let~$M$ be an algebraic subvariety of a \emph{meromorphic} stratum cut out by linear equations in period coordinates with {\em real} coefficients.  Let $\calC=\{C_1,\dots,C_d\}$ be an equivalence class of $M$-parallel horizontal cylinders on some~$(X,\omega)\in M$. Then for any $t,s\in\RR$ the flat surface $a_t^\calC u_s^\calC(X,\omega)$ obtained by applying to each~$C_i$ the matrix
$$
 a_t\circ u_s \hbox{ with }a_t=\left(\begin{smallmatrix} 1&0\\ 0& e^{t}\end{smallmatrix}\right), u_s=\left(\begin{smallmatrix} 1&s\\ 0& 1\end{smallmatrix}\right),
$$
and leaving the rest of the flat surface unchanged, is also contained in~$M$.
\end{thm}
The above is a generalization of \cite[Theorem 5.1]{wright}, where it is assumed that~$M$ is an affine invariant manifold (living in a \emph{holomorphic} stratum). We stress that to prove our results, we use degeneration techniques, only working near the boundary $\pM$ in $\LMS$; note, though, that the theorem above applies at {\em any} point of~$M$, not necessarily close to its boundary. Essentially what happens is that we know that~$M$ is cut out by linear equations near every of its points, and by analyzing the behavior of these equations near $\pM$ we obtain sufficiently many necessary conditions on these equations in order to control deformations at every point of~$M$.


\subsection*{ The linear equations of affine invariant manifolds}
Affine invariant manifolds are linear subvarieties of \emph{holomorphic} strata, with all equations having \emph{real} coefficients. Equivalently, by the foundational results of Eskin-Mirzakhani-Mohammadi, combined with the result of Filip, these are the (topological) closures of orbits of the $\GL^+(2,\RR)$ action on the holomorphic strata. These $\GL^+(2,\RR)$ orbits come up naturally in the study of billiards on rational polygons and the Teichm\"uller geodesic flow.

In this more restricted context of most interest we are able to obtain further information, similar to some results of Mirzakhani-Wright \cite{miwr}.  Of fundamental importance for us is the result of Avila, Eskin, M\"oller~\cite{aem} that for affine invariant manifolds, the tangent space projected to absolute homology is symplectic. This gives a way to use our precise understanding of relations among periods over vanishing cycles, given by~\Cref{thm:Mrltd} to obtain further results on \thc equations. Our strongest result in this direction is for affine invariant manifolds in the minimal stratum:

\begin{thm}\label{thm:AIM}
If $M$ is an affine invariant manifold in the minimal stratum ~$\Omega\calM_{g,1}(2g-2)$, then
\begin{enumerate}
\item The space of defining equations of~$M$ is spanned by defining equations that cross at most two horizontal nodes.
\item The space of defining equations of~$M$ that are linear combinations of periods over horizontal vanishing cycles is spanned by defining equations that are pairwise proportionalities of vanishing cycles.
\label{item:cross}
\end{enumerate}
\end{thm}
The two statements will follow from ~\Cref{prop:pairwise-cross} and \Cref{prop:pairwise-circum}, respectively.  Note that~\eqref{item:cross} above is the same statement as that of \Cref{thm:divisorial}, for the case of horizontal divisorial degenerations --- but in the context of affine invariant manifolds of the minimal stratum we prove it for arbitrary degenerations.

This precise description does not directly generalize to the case of the general stratum~$\omodulin(\mu)$, as we demonstrate in~\Cref{exa:counter-pairwise-cross} and~\Cref{exa:counter-pairwise-circum}. The main difficulty in discovering a suitable general statement lies in the fact that while vanishing cycles are naturally elements of absolute homology group $H_1(X;\ZZ)$, the defining equations naturally lie in $H_1(X,\zeroes;\CC)$ (recall that we are in the holomorphic case, so $\poles=\emptyset$).
We will investigate this general situation and application to classification of affine invariant manifolds in further work.

We also record some special properties of the boundary stratification of affine invariant manifolds in \Cref{prop:noncompact} and \Cref{cor:hashorizontal}.

\subsection*{Outline of the paper}
\label{sec:outline-paper}

\begin{itemize}
\item In~\Cref{sec:notation} we recall the moduli space of multi-scale differentials, describe the setup and notation for our study of linear subvarieties via degenerations, and recall the relevant machinery and results of the first author from~\cite{fred}.
\item  In~\Cref{sec:lineq} we start by studying irreducible components of the boundary~$\pM$, proving~\Cref{thm:divisorial}, and then use this recursively to prove~\Cref{thm:Mrltd}. We then further study \thc and non-horizontal equations in detail, proving~\Cref{thm:vertres} and~\Cref{thm:decomp}. Our most precise result is \Cref{prop:eqsquantitative}, which gives  the coefficients of defining equations.
\item In~\Cref{sec:transverse}, we use these detailed results to focus on the non-emptiness and dimensions of the strata~$\pMG$, converting linear equations to equations in plumbing coordinates to obtain~\Cref{thm:localequations}. The form of the equations in plumbing coordinates yields~\Cref{thm:transverse}, and allows us to construct a smooth compactification of Hurwitz spaces in~\Cref{prop:HW}.
\item In~\Cref{sec:cyldef} we further analyze the linear equations and degenerations to prove~\Cref{thm:cyldeformation}, our generalization of the cylinder deformation theorem.
\item  Finally, in~\Cref{sec:AIM} we specialize to the case of affine invariant manifolds. By~\cite{aem}, the tangent space of an affine invariant manifold, projected to absolute homology, is symplectic.  We use this to prove~\Cref{prop:pairwise-cross} and \Cref{prop:pairwise-circum}, which together constitute \Cref{thm:AIM}.
\end{itemize}

\subsection*{Acknowledgments}
We are grateful to Martin M\"oller for useful discussions and for sharing with us the results of the ongoing work of M\"oller and Mullane on related topics. We would also like to thank Alex Wright for many helpful discussions and useful feedback, and special thanks to Scott Mullane for suggesting the possibility of approaching Hurwitz spaces, as we do in~\Cref{prop:HW}.

\section{Notation and setup}\label{sec:notation}
In this section we recall those aspects of the setup, construction, and results of~\cite{BCGGMmsds} that we need, and also the setup and results of~\cite{fred}.  The most technical aspect of ~\cite{fred}, log period spaces, is not necessary for our study.
\subsubsection*{Level graphs and multi-scale differentials}
For a stable Riemann surface $(X,\underline{x})\in\modulin$ we denote $\Gamma$ its dual graph. We will follow the convention of~\cite{BCGGMmsds} in always suppressing the notation for marked points, unless they are used explicitly. A {\em level graph} structure $\overline\Gamma$ on $\Gamma$ is given by a function $\ell:V(\Gamma)\surj\{0,-1,\dots,-L(\lG)\}$. For any $i\in \{0,-1,\dots,-L(\lG)\}$ the subgraphs $\overline{\Gamma}_{(i)}, \lG_{(<i)}$ and so on are defined by taking the induced subgraph on the set of all vertices $v\in V(\Gamma)$ such that $\ell(v)=i$, or $\ell(v)<i$ and so on. For example, the vertices of $\overline{\Gamma}_{(i)}$ are all vertices at level $i$, and the edges are those edges $e\in E(\Gamma)$ such that both their endpoints lie at level~$i$.

An edge $e\in E(\overline\Gamma)$ is called {\em horizontal} if it connects two vertices of the same level, and called {\em vertical} otherwise, and we write $E(\overline\Gamma)=\Ehor[\overline\Gamma]\sqcup\Ever[\overline\Gamma]$. For a vertical edge~$e\in\Ever[\overline\Gamma]$ we denote by $\lbot$ and $\ltop$ the levels of its bottom and top vertex, respectively. We denote $\Ehori\subseteq\Ehor$ the set of horizontal edges connecting vertices of level~$i$.  A multi-scale differential is the data of a stable Riemann surface $X$ together with a collection $\eta=\{\eta_v\}$ of meromorphic differentials on the irreducible components $X_v$ of~$X$ satisfying various conditions described in~\cite{BCGGMivc,BCGGMmsds} --- in particular~$\eta$ has simple poles at all horizontal nodes. An enhancement $\overline\Gamma^+$ of a level graph is a choice of a positive integer $\kappa_e$ for every vertical edge $e$. This integer prescribes the order of zero of the multi-scale differential to be $\kappa_e-1$ at the top preimage of the node~$e$, and the order of pole to be $\kappa_e+1$ at the bottom preimage of~$e$. We will always work with level graphs with a chosen and fixed enhancement, but to keep the notation manageable, from now will simply write $\lG$ for an enhanced level graph. Additionally, the data of a multi-scale differential includes a prong-matching, and great care is needed in understanding equivalence of multi-scale differentials, but as we will be working locally on $\LMS$, we will be able to mostly avoid these considerations.

\subsubsection*{Undegenerations and plumbing}
The boundary $\partial\LMS$ is stratified. It is convenient for us to denote $\DG$ the {\em open} boundary strata (note that in~\cite{euler} this notation is used for closed boundary strata) indexed by enhanced level graphs.  A stratum~$\DG$ is essentially a finite union of some finite covers of products of linear subspaces of products of some strata of meromorphic differentials; in particular~$\DG$ may be disconnected (see~\cite[Sec.~4]{euler} and \Cref{rem:genstratadivisorial} below for more discussion). All of our constructions will be performed locally in a neighborhood $U$, that we will now describe, of  a chosen fixed point $p_0=(X_0,\Gamma,\eta_0)\in\DG$.

The codimension of a stratum $\codim_{\LMS}\DG$ is equal to $H(\lG)+L(\lG)$, where ~$H(\lG):=\#\Ehor$. Fix a small open neighborhood $p_0\in W\subset\DG$. Then a neighborhood of $p_0$ in $\LMS$ can be given as $U:=W\times \Delta^{H(\lG)+L(\lG)}$, where $\Delta$ is a sufficiently small complex disk around zero. Coordinates on the second factor are called plumbing coordinates, which we denote $\{h_e\}_{e\in\Ehor}$ and $\{t_i\}_{i\in \{-1,\dots,-L(\lG)\}}$. We will denote  $U^\circ:=W\times (\Delta^*)^{H(\lG)+L(\lG)}$ the set of all smooth flat surfaces in~$U$. From now on, when we speak of~$U$ and~$W$, we will allow ourselves to further shrink the neighborhoods as necessary.

An open stratum $\DG[\lG']$ intersects $U$ if and only if the (enhanced) level graph $\lG'$ is an undegeneration of $\lG$ (which we write as $\lG'\rightsquigarrow\lG$).  Equivalently, there is a simplicial graph morphism $dg:\lG\to\lG'$, which is obtained as a composition $dg=dg^\hor\circ dg^\ver$ of the horizontal undegeneration $dg^\hor$ that only contracts some set of horizontal edges, and a vertical undegeneration that only contracts some set of level transitions. We refer to~\cite{BCGGMmsds} for a discussion of the behavior of enhancements and the (very delicate) behavior of prong-matchings under undegeneration. Explicitly, the closure $\overline{\DG[\lG']}\cap U$ is the coordinate subspace of~$U$ given by equations $h_e=0$ for all $e\in \Ehor[\lG']\subseteq \Ehor$ and $t_i=0$ for all level transitions of $\lG$ that persist in $\lG'$. From now on, whenever we speak of an undegeneration $\lG'$, we implicitly mean with a {\em given} graph morphism $dg:\lG\surj\lG'$.

Any flat surface $p=(X,\omega)\in U^\circ$ can be obtained by plumbing some $(X_b,\eta_b)\in W\subset\DG$. The plumbing procedure replaces a neighborhood of each node $e\in X_b$, which is locally a union of two disks identified at the origin, with a cylinder, suitably glued to the rest of the surface. We denote $\Lambda_e\subset X$ the pinching curve, also called the seam, which is the circumference curve of this cylinder. The {\em vanishing cycle} is the homology class $\van:=[\Lambda_e]\in\relhomZ[X][\poles][\zeroes]$. We recall that $H_1(X\setminus \zeroes;\ZZ)\inj \relhomZ[X][\poles][\zeroes]$ (thinking of the vanishing cycles in relative, rather than absolute, homology, will be essential in \Cref{sec:AIM}). The intersection pairing $\relhomZ[X][\poles][\zeroes]\times \relhomZ\to\ZZ$ then allows us to compute intersection numbers of $\van$ with elements of $\relhomC$. We note that $\lambda_e$ is only defined up to sign; most of our formulas will include $\van$ with coefficient proportional to the intersection number $\langle \gamma,\van\rangle$ for some~$\gamma\in\relhomC$, which will eliminate this sign ambiguity.
Cutting $X$ along the multicurve $\Lambda:=\{\Lambda_e\}_{e\in E(\lG)}$ decomposes the smooth Riemann surface~$X$ into the union $X_{(0)}\cup X_{(-1)}\dots\cup X_{(-L(\lG))}$ of its levelwise pieces, where the pieces intersect along the seams $\Lambda_e$ for $e\in\Ever$.

While, as discussed above, the local coordinates on $\LMS$ transverse to $\DG$ are given by $t_{-1},\dots,t_{-L(\lG)}$ and $\{h_e\}_{e\in\Ehor}$, the plumbing coordinates $s_e$ for vertical nodes are related to $t_i$ by the equation
\begin{equation}\label{eq:se}
s_e^{\kappa_e}=\prod_{i=\lbot}^{\ltop-1} t_i^{m_{e,i}},
\end{equation}
where we recall that by definition~\cite[(6.7)]{BCGGMmsds}, $a_i$ is the least common multiple of $\kappa_e$ for all $e\in\Ever$ such that $\ltop>i\ge\lbot$, and $m_{e,i}:=a_i/\kappa_e$.

\subsubsection*{The boundary neighborhood in~$M$}
Throughout the text, we will fix once and for all a linear subvariety $M\subseteq\omodulin(\mu)$ of codimension~$m$, and will consider its closure $\oM\subseteq\LMS$, so that the projectivization $\PP\oM\subseteq \PP\LMS$ is compact. Since $\oM$ is the closure of an algebraic subvariety $M\subseteq\omodulin(\mu)$ in the algebraic compactification $\LMS$ of~$\omodulin(\mu)$, it follows that~$\oM$ is algebraic~\cite[Cor.~10.1]{mumfordredbook}.

We will choose $p_0\in\pM\cap\DG$ and for any undegeneration $\lG'\rightsquigarrow\lG$ denote $\pMG[\lG']:=\pM\cap U\cap\DG[\lG']\subset\DG[\lG']$. Recall that in general~$M$ is an immersed submanifold of~$\omodulin$, and we will always want to work at a flat surface~$p=(X,\omega)$ that is a smooth point of $M\cap U$, to avoid having to deal with~$M$ having multiple local irreducible components at~$p$ (each linear in period coordinates). We use~$(X',\lG',\eta')$ to denote points $(X',\omega')\in\pMG[\lG']$ on the (local) open strata corresponding to undegenerations $\lG'\rightsquigarrow\lG$. We will  often omit $\omega$ or $\omega'$ in our notation for flat surfaces.

\subsubsection*{Top level of paths and homology classes}
In~\cite{fred} the first author determined the defining equations for $\pMG\subseteq\DG$ at~$p_0$, in (generalized) period coordinates on~$\DG$,  starting from the defining equations for $M\subseteq\omodulin(\mu)$ at a nearby point~$p\in M\cap U$. Qualitatively, the result is that $\pMG$ is given by linear equations on $\DG$, but we will need the precise description of these equations, which we now recall.

To state the results of~\cite{fred}, we need to restrict paths in~$X$ to their top level. The top level $\topl(\beta)$ of any collection of paths $\beta\subset X$ is the largest $i$ such that $\beta\cap X_{(i)}\ne\emptyset$. The {\em top level $\topl([\beta])$} of a class $[\beta]\in \relhomC$ is the minimum of $\topl(\beta)$ over all collections of paths $\beta$ representing the class~$[\beta]$. For a homology class $[\beta]$, we define its {\em top level restriction} $\toplr$ to be the element $\toplr\in\hlgr[\topl(\beta)]$ defined by choosing a collection of paths $\beta$ representing $[\beta]$ such that $\topl(\beta)=\topl([\beta])$, and restricting each path in~$\beta$ to $X_{(\topl([\beta])}$, considered as a relative homology class there. In~\cite[Prop. 4.2]{fred}, it is shown that this is well-defined. See \Cref{fig:toplvl} for an illustration.

\begin{figure}
\includegraphics[scale=0.75]{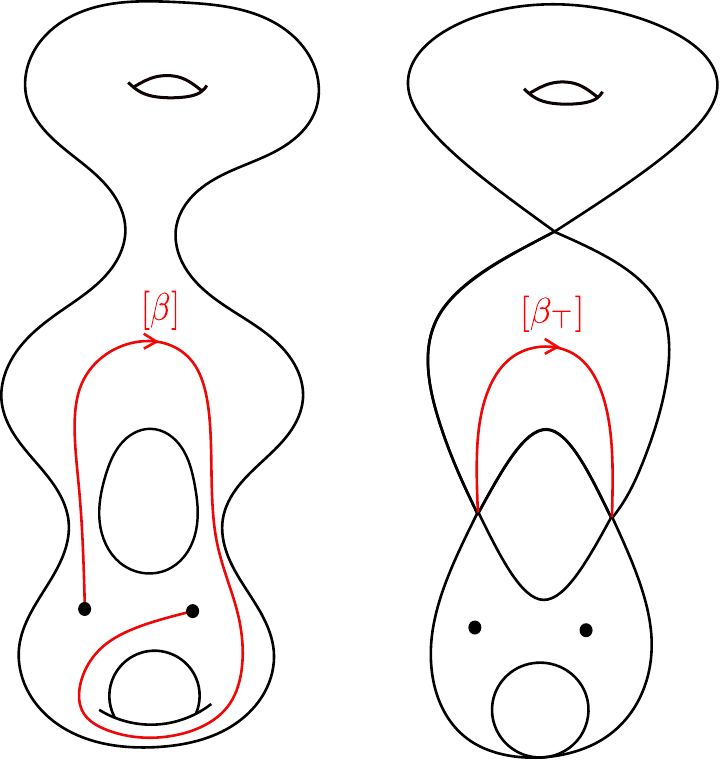}
\caption{The top level restriction of a homology class}
\label{fig:toplvl}
\end{figure}

For a smooth flat surface $p\in U^\circ$ a homology class  $[\beta]\in \relhomC$ is said to be {\em crossing} a node $e\in E(\lG)$ if $\langle [\beta],\van\rangle\neq 0$, where recall that we think of $\van$ as an element of~$\relhomZ[X][\poles][\zeroes]$. We call $[\beta]$ a {\em \hcc} if it crosses some horizontal vanishing cycle {\em at level $\topl([\beta])$}. To simplify language, we will call {\em \nhc} any class~$[\beta]$ that is not a \hcc, and emphasize that such a \nhc~$[\beta]$ may still intersect horizontal vanishing cycles at levels strictly below $\topl([\beta])$.

\subsubsection*{A $\lG$-adapted basis}
We recall from~\cite{fred} that a $\lG$-adapted basis  is a basis for $\relhomZ$ satisfying the following properties. First, all of its elements that are \hccs have intersection 1 with $\van$ for a unique $e\in\Ehor$, where these~$e$ are distinct for different \hccs in the basis, and do not cross any other horizontal nodes. Elements of a $\lG$-adapted basis such that their top level is $i$ can be listed as
$$
\left\{ \crC_1,\dots,\crC_{\lgcr},\hfC_1,\dots,\hfC_{\lghf}\right\},
$$
where each $\crC_j$ is a \hcc with $\langle \crC_j,\van[e_{j}^{(i)}]\rangle=1$ for some distinct horizontal node $e_j^{(i)}\in\Ehori$, and such that $\crC_j$ does not cross any other horizontal nodes at any level. Furthermore, the definition of being a $\lG$-adapted basis requires that each $\hfC_j$ does not cross any horizontal nodes, at any level, and that for any $i$ the top level restrictions $\left\{\toplr[(\hfC_1 )],\dots,\toplr[(\hfC_{\lghf})]\right \}$ form a basis of the quotient of $\hlgr$ by the subspace of Global Residue Conditions. The existence of a $\lG$-adapted basis for any~$\lG$ is proven in~\cite[Prop.~4.8]{fred}. Sometimes we do not need to specify the level of the homology classes and whether they cross horizontal nodes or not. In this case we write the $\lG$-adapted basis simply as
\begin{equation}\label{eq:lGadaptedbasis}
\{ \adC_1,\dots, \adC_{\lgad}\} = \bigsqcup_{i= -L(\lG)}^0 \left\{ \crC_1,\dots,\crC_{\lgcr}, \hfC_1,\dots,\hfC_{\lghf}\right\},
\end{equation}
where $\lgad:= \dim \relhomC=\dim\LMS$. We will choose and fix a $\lG$-adapted basis from now on.

\subsubsection*{Defining equations of~$M$}
The technical core of our arguments is investigating the linear equations for $\pMG$. To keep the notation manageable, we simply say that $F\in\relhomC$ is a {\em defining equation of~$M$} if $\int_{F}\omega=0$ holds identically on~$M$ in a neighborhood of a fixed chosen flat surface~$p\in M\cap U$. We will denote $N\subseteq \relhomC$ the linear space of all defining equations of~$M$ at~$p$, denoted thus because it is the normal space in period coordinates. As discussed in the introduction, the space $N$ is locally constant along~$M$ near~$p$, and thus throughout the paper we should be carefully treating irreducible components $Z$ of $M\cap U$ (which, after shrinking~$U$, are in bijection with the local irreducible components of $\oM$ at $p_0$) individually. To keep the notation and language manageable, we will just speak of defining equations, making the discussion of local irreducible components precise in~\Cref{sec:transverse}, where it is crucial.

Denote $C_l\in\CC$ the coefficients of $F$ in our fixed $\lG$-adapted basis $\{\gamma_l\}_{l=1,\ldots,K}$, so that
\begin{equation} \label{eq:Fformsimple}
F(X,\omega)= \sum_{l=1}^{\lgad} C_{l}\int_{\gamma_l} \omega.
\end{equation}
Equivalently, writing out the basis elements separately, we denote the coefficients of $F$ by $A_l^{(i)},B_l^{(i)}\in\CC$, so that
\begin{equation}\label{eq:Fform}
 F(X,\omega)=\sum_{i=-L(\lG)}^{\topl(F)}\left(\sum_{l=1}^{\lgcr}A_l^{(i)}\int_{\crC_l}\omega
 +\sum_{l=1}^{\lghf}B_l^{(i)}\int_{\hfC_l}\omega\right).
\end{equation}

Writing down all defining equations of~$M$ at~$p$ involves a choice of the basis of the vector space~$N$. We will always choose a basis of defining equations such that the matrix $C=(C_{kl})$ of the coefficients of defining equations~\eqref{eq:Fformsimple} is in {\em reduced row echelon form (rref)} with respect to our chosen $\lG$-adapted basis, and denote $F_1,\dots,F_m$ such a {\em rref basis}.

\subsubsection*{Equations of~$\pMG$ from equations of~$M$}
In~\cite{fred}, the main quantitative result is a way to read off the  equations for $\pMG\subseteq\DG$ from a rref basis:
\begin{thm}[{\cite[Thm.~1.2 and Prop.~8.2]{fred}}]\label{thm:fred}
For each $j=1,\dots,m$, if $F_j$ is a non-horizontal cycle, let $G_j:=\toplr[(F_j)]$, and if $F_j$ is a \hcc, then let $G_j:=0$. Then $G_1,\dots,G_m$ are a basis for the space of local defining equations for $\pMG$ within $\DG$.
\end{thm}
Essentially what this says is that we represent each equation~$F_j$ by a collection of paths such that its top level is minimal possible, equal to $\topl(F_j)$; then if this collection of paths crosses any horizontal node at its top level, then on~$\DG$ we ``lose'' this defining equation~$F_j$, otherwise the equation $F_j$ on~$\DG$ yields the equation  $\toplr[(F_j)]$.

In view of this theorem, for our fixed linear subvariety~$M$ and for any undegeneration~$\lG'\rightsquigarrow\lG$ we denote by~$c(\lG')$ the number of defining equations of~$M$ at $p$ that are lost on $\DG[\lG']\cap U$. Thus the number of defining equations for $\pMG[\lG']$ inside $\DG[\lG']$ is equal to $m-c(\lG')$, and thus ~\Cref{thm:fred} implies that
\begin{equation}\label{eq:codim}
\codim_{\LMS}(\pMG[\lG'])=H(\lG')+L(\lG')+m-c(\lG'),
\end{equation}
since $\pMG[\lG']$ has codimension $m-c(\lG')$ within the open stratum $\DG[\lG']$, which itself has codimension $H(\lG')+L(\lG')$ in~$\LMS$.
For further use, we call a degeneration $dg:\lG\to\lG'$ {\em divisorial} if $\dim_\CC\pMG=\dim_\CC\pMG[\lG']-1$.
\begin{rem}
\label{rem:lift}
A consequence of \Cref{thm:fred} is that linear equations for $\pMG$ can be {\em lifted} to $M$. More precisely,
if~$F$ is a linear equation among periods which is satisfied on $\pMG$ in a neighborhood of $p_0$, and $F$ is completely contained in level $i$, i.e. $F$ can be represented by paths contained in $X_{(i)}$, then there exists a linear equation $G$ for $M$, valid in a neighborhood of a nearby point $(X,\omega)\in M$, such that $G_{\topl}=F$. In other words,~$F$ is the top level restriction of the linear equation $G$.
We stress that one can only lift an equation~$F$ for $\pMG$ if it is completely contained in a fixed level. Any linear equation defining $\pMG$ can then be written as a sum of linear equations, each of which is completely contained in some level (these levels might be different for different summands), and each of these summands can be lifted.
\end{rem}

\section{Degenerations of linear equations}\label{sec:lineq}
In what follows, given a defining equation~$F$ of~$M$ at~$p$, it will be useful to consider various associated periods. For any  undegeneration $\lG'\rightsquigarrow \lG$ and for any collection of integers $\{n_e\colon e\in E(\lG')\}$, we define the {\em $(\lG',\underline{n})$-residue of $F$} by
\begin{equation}\label{eq:residuedefine}
R(F,\lG',\underline{n})\colon= \sum_{e\in E(\lG')}n_e\langle F,\lambda_e\rangle \int_{\lambda_{e}}\omega.
\end{equation}
\begin{prop}[The monodromy argument, {\cite[Prop. 7.6]{fred}}]\label{prop:monodromy}
For any defining equation $F$ of~$M$ at~$p$, if $\pMG[\lG']$ is non-empty for some undegeneration $\lG'\rightsquigarrow\lG$, then for some collection  $\underline{n}$ of {\em positive} integers the residue $R(F,\lG',\underline{n})$ is identically zero on~$M$.
\end{prop}
For the convenience of the reader, we quickly recall from~\cite{fred} the outline of the proof.
\begin{proof}
Let $f:\Delta\to \oM$ be a holomorphic map from a disk such that $f(\Delta^*)$ is contained in the smooth locus of~$M$, $p\in f(\Delta)$, and $f(0)=p_0$. We define $n_e$ to be the (positive, since $s_e(p_0)=0$) vanishing order of $s_e\circ f$ at $z=0\in\Delta$, where we recall that $s_e$ is the plumbing parameter for the corresponding node. Using Picard-Lefschetz, the monodromy of any cycle $[\beta]\in\relhomZ$ for the Gauss-Manin connection along $f$ can be computed in terms of vanishing cycles as
\[
[\beta]\mapsto [\beta] + \sum_{e\in E(\lG)}n_e\langle \gamma,\lambda_e\rangle\lambda_e.
\]
Thus parallel transport along the generator of $\pi_1(\Delta^*)$ transforms the equation $F$ written in the form~\eqref{eq:Fformsimple} into
\[
\sum_{l=1}^{\lgad} C_l\left (\int_{\adC_l}\omega + \sum_{e\in E(\lG)}n_e\langle \gamma_l,\lambda_e\rangle\int_{\lambda_e}\omega\right),
\]
which must then also be a defining equation of~$M$ at $p$, and then subtracting $F$ in the form~\eqref{eq:Fformsimple} from this equation yields the Proposition.
\end{proof}

We now use this monodromy argument to quickly prove the necessary conditions for boundary strata of $\LMS$ to contain irreducible components of the boundary~$\pM$. Since $\oM\subseteq\LMS$ is an algebraic subvariety, its intersection~$\pM$ with the boundary of $\LMS$, which is a divisor, is an equidimensional variety of dimension $\dim_\CC\pM=\dim_\CC M-1$. For~$U\ni p_0$ sufficiently small, each irreducible component~$Y$ of $\pM\cap U$ must contain~$p_0$. Thus the generic point of~$Y$ must be contained in the open stratum $\DG[\lG']$ for some undegeneration $\lG'\rightsquigarrow \lG$.

\begin{proof}[Proof of~\Cref{thm:divisorial}]
Denote $Y^o:=Y\cap {\DG[\lG']}$ the open part of a divisorial boundary component, so that $\dim_\CC Y^o=\dim_\CC M-1$. Substituting the dimension of $Y^o$ from~\eqref{eq:codim} yields
\begin{equation}\label{eq:codimcompute}
 H(\lG')+L(\lG')+m-c(\lG')=m+1.
\end{equation}
Recall that $c(\lG')$ is the number of equations~$F_j$ in the rref basis that are ``lost'' on~${\pMG[\lG']}$, which are those where $F_j$ is \thc.
\Cref{prop:monodromy} shows that every  equation $F_j$ crosses at least two horizontal nodes (otherwise if it only crossed one horizontal node, then the period of~$\omega$ over the corresponding vanishing cycle would vanish, which is impossible, since the twisted differential must have a simple pole at every horizontal node).  By definition, for each \thc equation $F_j$ of the rref basis the pivot corresponds to a horizontal node $\lambda_e$ crossed by $F_j$.  Thus either $H(\lG')=0$ or $c(\lG')\leq H(\lG')-1$. In the first case we conclude from~\eqref{eq:codimcompute} that $L(\lG')=1$.

In the latter case substituting $H(\lG')\ge c(\lG')+ 1$ into the left-hand-side of ~\eqref{eq:codimcompute} yields $c(\lG')+1+L(\lG')+m-c(\lG') \ge m+1$, which is only possible if $L(\lG')=0$, and moreover if $H(\lG')=c(\lG')+1$. In that case the set of equations of the rref basis of defining equations can cross only one additional horizontal node in addition to the pivots, and the matrix, in rref, of defining equations of~$M$ must have the form
\[
 \left(\begin{array}{@{}c|c@{}}
    \begin{matrix}
    1 & 0 &\dots &0 & d_1   \\
    0 & \ddots & \vdots& \vdots&\vdots \\
     \vdots& 0 & 1& 0 & d_{H(\lG')-1} \\
   0& 0 & 0& 1 & d_{H(\lG')}
   \end{matrix} &\makebox[7em]{\BigStar}\\
	 \hline
\raisebox{-1.5em}{\makebox[7em]{ \BigZero}}& \raisebox{-1.5em}{\makebox[7em]{\BigStar}}\\ [3em]
 \end{array}\right)
\]
where the upper rows correspond to \thc equations for $\lG'$, the $d_l$ are non-zero, and the lower rows correspond to non-horizontal equations for $\lG'$.
\end{proof}

\subsection{Generalized strata, and constructing degenerations recursively}
In the proof of~\Cref{thm:Mrltd} below, and for potential applications of our machinery to classifying or ruling out existence of linear subvarieties of a given stratum, one needs to apply \Cref{thm:divisorial} recursively. Starting from a non-complete linear subvariety $M\subseteq \omodulin(\mu)$, we consider a divisorial boundary component $M':=\pMG\subseteq \DG$, for which~\Cref{thm:divisorial} gives necessary conditions on the graph $\Gamma$. By the results of~\cite{fred} $M'$ is locally given within $\DG$ by linear equations, and we would like to apply \Cref{thm:divisorial} again to yield a further divisorial degeneration of $M'$, assuming again that $M'$ is non-complete. However, \cite[Thm.~1]{fred} as stated does not apply to show that $\partial M'_{\lG'}$ inside of $\overline{\DG[\lG']}$ is a linear subvariety, because in general $\overline{\DG[\lG']}$ will be singular.
\begin{rem}[Generalized strata of differentials]\label{rem:genstratadivisorial}
The stratification of the boundary of $\LMS$ is discussed in detail in~\cite[Sec.~4]{euler}. The boundary strata are called there {\em generalized} strata of differentials. We now recall their geometric description and explain how our results can be adapted to this generality.

Let $\bm{g}=(g_1,\ldots,g_k)$ be  a tuple of genera, $\bm{n}=(n_1,\ldots,n_k)$ a tuple of positive integers, and $\bm{\mu}=(\mu_1,\ldots,\mu_k)$ a tuple of types of differentials, i.e.~$\mu_i$ is a partition of $2g_i-2$ into (not necessarily positive) integers, of length $n_k$.
The disconnected stratum is defined to be
\[
\Omega\calM_{\bm{g},\bm{n}}(\bm{\mu}):=\prod_{i=1}^{k} \Omega\calM_{g_i,n_i}(\mu_i),
\]
and the projectivization $\PP\Omega\calM_{\bm{g},\bm{n}}(\bm{\mu})$ is the quotient of $\Omega\calM_{\bm{g},\bm{n}}(\bm{\mu})$ by the diagonal $\CC^*$-action. A residue subspace $\mathfrak{R}$ is a set of linear equations on residues, modeled on the global residue conditions and matching residue conditions, see~\cite[Sec.~4.1]{euler} for the precise definition. The  generalized stratum $\Omega\calM_{\bm{g},\bm{n}}^{\mathfrak{R}}(\bm{\mu})$ modeled on a residue subspace $\mathfrak{R}$ is the subspace of $\Omega\calM_{\bm{g},\bm{n}}(\bm{\mu})$ consisting of all surfaces with residues lying in $\mathfrak{R}$.
In \cite[Prop. 4.2]{euler} the authors construct a compactification $\PP\Omega\overline{\calM}_{\bm{g},\bm{n}}^{\mathfrak{R}}(\bm{\mu})$ of $\PP\Omega\calM_{\bm{g},\bm{n}}^{\mathfrak{R}}(\bm{\mu})$ similar to the moduli space of multi-scale differentials.
For an enhanced level graph $\lG$  and for each level $i$, let
$(\bm{g}^{[i]},\bm{n}^{[i]},\bm{\mu}^{[i]},\mathfrak{R}^{[i]})$ be the tuple consisting of the genera, number of points, types of differentials and residues conditions at each irreducible component of level $i$.
The {\em generalized stratum} associated to $\lG$ is
\[
B_{\lG}:= \Omega\calM_{\bm{g}^{[0]},\bm{n}^{[0]}}^{\mathfrak{R}^{[0]}}(\bm{\mu}^{[0]}) \times\prod_{i=-L(\lG)}^{-1} \PP\Omega\calM_{\bm{g}^{[i]},\bm{n}^{[i]}}^{\mathfrak{R}^{[i]}}(\bm{\mu}^{[i]})
\]
and, by replacing  $\PP\Omega\calM_{\bm{g}^{[i]},\bm{n}^{[i]}}^{\mathfrak{R}^{[i]}}(\bm{\mu}^{[i]})$ with $\PP\Omega\overline{\calM}_{\bm{g}^{[i]},\bm{n}^{[i]}}^{\mathfrak{R}^{[i]}}(\bm{\mu}^{[i]})$, we define $\overline{B}_{\lG}$ similarly.
The generalized stratum $B_{\lG}$ admits a system of {\em generalized period coordinates} as described in \cite[Sec. 2.6]{fred}, with transition functions that are linear on the top level, and projective-linear on lower levels.

In~\cite{euler} the following diagram is constructed
\[
\begin{tikzcd}
&&D_{\lG}^s\arrow{dl}[swap]{p_{\Gamma}} \arrow{dr}{c_{\Gamma}} & &\\
\overline{B}_{\Gamma} & B_{\Gamma}\arrow[hook']{l}& & D_{\Gamma}\arrow[hook]{r} & \overline{D}_{\Gamma}
\end{tikzcd}
\]
where $c_{\Gamma},p_{\Gamma}$ are covering maps. We do not give the precise definition of $D^s_{\gamma}, c_{\Gamma},p_{\Gamma}$ and instead refer the reader to Section $4$ in (loc.cit.).
\end{rem}

The main theorem of \cite{fred} can then be rephrased as
\begin{prop}Let $M\subseteq \Omega\calM_{g,n}(\mu)$ be a linear subvariety. Then
\[
p_{\lG}(c_{\lG}^{-1}(\pMG))\subseteq B_{\lG}
\]
is a {\em levelwise} linear subvariety for the linear structure on $B_{\lG}$.
\end{prop}
By abuse of notation we will just write $\pMG$ for $p_{\lG}(c_{\lG}^{-1}(\pMG))$.
Levelwise here means that the linear equations defining $\pMG$ only restrict periods of the same top level.
Now given a linear subvariety $M'$ of $B_{\lG}$ and a boundary stratum $D_{\lG'}$ of $\overline{B_{\lG}}$ we proceed in the same way and can thus consider $\partial M'_{\lG'}\subseteq \DG[\lG
']$ also as a linear subvariety, by the same abuse of notation as above.

\begin{rem}[Constructing chains of divisorial degenerations]\label{rem:chain}
By the previous Remark, we can now use \Cref{thm:divisorial} to construct chains of undegenerations, where each is divisorial in the next.
Let~$M$ be a linear subvariety in a (possibly generalized) stratum $B$ and $p_0\in \pMG$ a boundary point. The undegenerations $\lG'\rightsquigarrow\lG$ corresponding to boundary divisors $\DG[\lG']\subseteq B$ are those where  $\lG'$ either has only two levels and no horizontal nodes or such that $\lG'$ has a unique edge, which is horizontal. In the former case the undegeneration $\lG'\rightsquigarrow\lG$ corresponds to keeping only those edges of $\lG$ that cross some level transition (i.e.~such that  $\ltop>i\ge\lbot$). Since $p_0\in\pMG\subseteq\overline{\pMG[\lG']}$, in either case the intersection of $\oM$ with $\overline{\DG[\lG']}$ is non-empty. Since it is an intersection with a divisor, it follows that $\dim \overline{\pMG[\lG']}=\dim M-1$.

If $D_{\lG'}$ is a purely vertical divisorial stratum, by \Cref{thm:divisorial}, a generic point of the intersection of $\oM$ with  $\overline{\DG[\lG']}$ is contained in the open boundary stratum $\DG[\lG']$.

On the other hand, if $D_{\lG'}$ is a purely horizontal boundary stratum, again by \Cref{thm:divisorial} there exists an intermediate undegeneration $\lG'\rightsquigarrow\lG''\rightsquigarrow\lG$ such that some irreducible component of $\partial M$ is generically contained in $\pMG[\lG''],\dim_\CC\pMG[\lG'']=\dim_\CC M-1$ and furthermore $\lG''$ is a purely horizontal level graph.

We can thus construct chains of undegenerations
\[
pt= \Gamma_0 \rightsquigarrow  \Gamma_1 \rightsquigarrow \cdots \rightsquigarrow \Gamma_d= \Gamma,
\]
such that each boundary stratum $\pMG[\lG_i]$ is non-empty, each undegeneration \[\Gamma_j \rightsquigarrow\Gamma_{j+1}\] is either purely vertical or purely horizontal and furthermore $\dim_\CC \pMG[\Gamma_{j+1}]=\dim_\CC \pMG[\Gamma_{j}]-1$. Note that $pt$, a single vertex and no edges, corresponds to the open stratum $B$ itself. Such a chain of divisorial degenerations is determined by prescribing at each step $j$ whether the undegeneration $\Gamma_j \rightsquigarrow\Gamma_{j+1}$ smoothes some given level transition (and nothing else), or by requiring $\Gamma_j \rightsquigarrow\Gamma_{j+1}$ to smooth a given horizontal edge. However, in that case the undegeneration may also have to smooth some further collection of horizontal edges.

Sometimes it will be more convenient to think of degenerations rather than undegenerations. When thinking of $\Gamma_j \rightsquigarrow\Gamma_{j+1}$  as a degeneration of $\Gamma_j$, instead of smoothing out a level transition or a collection of nodes we will then say that the degeneration pinches a level transition or a collection of nodes.
\end{rem}

\begin{proof}[Proof of \Cref{thm:vertres}]
Consider the closed boundary divisor $\overline{\DG[{\lG [i]}]}:=\{t_i=0\}\subseteq \partial\LMS$. In other words $\lG [i]$ is the undegeneration of $\lG$ opening up all horizontal nodes and all level passages except the one between level $i+1$ and level $i$. In particular $E(\lG [i]):=\{ e\in E(\lG)\,:\, \ltop> i\ge\lbot\}$. Since the intersection $\overline{\DG[{\lG [i]}]}\cap \oM$ is non-empty because it contains~$p_0$, this intersection is a divisor in $\oM$. Let~$Y$ be an irreducible component of $\pM$ contained in $\overline{\pMG[\lG [i]]}$. By \Cref{thm:divisorial}, $Y$ is generically contained in $\DG[{\lG [i]}]$. Let then $F$ be any defining equation of $M$, written in the form~\eqref{eq:Fformsimple}.  Since $\pMG[{\lG [i]}]$ is non-empty, it follows from \Cref{prop:monodromy} that there exist positive integers $n_e$ for $e\in E(\lG [i])$ such that
\begin{equation}\label{eq:Ri}
\sum_{e\in E(\lG [i])}n_e\langle F,\lambda_e\rangle\int_{\lambda_{e}}\omega  =0.
\end{equation}
We recall that the integers $n_e$ are computed as vanishing orders of the plumbing parameter $s_e$ along a degenerating family and thus  by \eqref{eq:se} there exists an integer $d$ such that
\[
n_e=d\cdot m_{e,i}
\] and thus \eqref{eq:Ri} is equivalent to
\begin{equation}\label{eq:vereq}
R_i(F):=\sum_{e\in E(\lG [i])}m_{e,i}\langle F,\lambda_e\rangle\int_{\lambda_{e}}\omega =0.
\end{equation}
\end{proof}

\subsection{\Mrel nodes, and proofs of~\Cref{thm:Mrltd} and~\Cref{thm:decomp}}
We now further investigate the form of linear equations crossing horizontal nodes, setting up the notation and preliminary results for the proof of~\Cref{thm:Mrltd} and ~\Cref{thm:decomp}.
\begin{df}\label{df:Mrelated}
For a defining equation $F$ we let
\[
\Ehor[F]:= \{ e\in \Ehor \,|\, \langle F,\van\rangle \ne 0\}.
\]
be the set of horizontal nodes crossed by~$F$. A set $S\subseteq\Ehor$ is called {\em $M$-correlated} if $S= \Ehor[F]$ for some defining equation~$F$ of~$M$ at~$p$. An $M$-correlated set $S$ is called {\em $M$-primitive} if no proper subset of $S$ is $M$-correlated.
We let $\sim$ be the equivalence relation on $\Ehor$ generated by $M$-primitive subsets, and if $e\sim e'$ we say $e,e'$ are {\em \Mrel}. In words, two nodes $e,e'\in \Ehor$ are \Mrel  if there exist $M$-primitive collections $S_1,\dots,S_k$ of horizontal nodes, and a sequence $e=e_0,\dots, e_k=e'$ of horizontal nodes such that $\{e_i,e_{i+1}\}\subseteq S_{i+1}$ for $i=0,\dots, k-1$. The relation $\sim$ partitions $\Ehor$ into {\em \Mequiv classes}.
\end{df}
\begin{rem}\label{rem:limres}
The purpose of this definition is to formalize the notion that there is a defining equation~$F$ that crosses both nodes, and that cannot be written as a sum of two defining equations that each cross a strictly smaller collection of horizontal nodes. Note that the definition does not require \Mrel nodes to be of the same level, but as periods of~$\omega$ over horizontal vanishing cycles of different levels go to zero at different rates, we will see below in~\Cref{cor:Mrltdlvl} that \Mrel horizontal nodes must in fact have the same level.
\end{rem}

We now show that \Mequiv classes can be computed using the rref basis $F_1,\dots,F_m$. We say two  nodes $e,e'\in \Ehor$ are {\em \rrefrel}, if there exists a chain of elements $F_{l_1},\dots, F_{l_k}$ of the rref basis, and a sequence $e=e_0,\dots, e_k=e'$ of horizontal nodes such that $\lbrace e_i,e_{i+1}\rbrace\subseteq\Ehor[F_{l_{i+1}}]$, for each $i=0,\dots, k-1$. Said differently, rref-cross-equivalence is the equivalence relation generated by $\Ehor[F_1],\dots,\Ehor[F_m]$.

\begin{lm}\label{lm:Mrref}
Two horizontal nodes are \Mrel if and only if they are \rrefrel.
\end{lm}
\begin{proof}
We first claim that the set $\Ehor[F_j]$ is $M$-primitive for each $j=1,\dots,m$.
Assume for contradiction that there exists a defining equation $F$ with $\Ehor[F]\subsetneq \Ehor[F_j]$. If we write $F=\sum_{l=1}^{m} a_lF_l$ in terms of the rref basis, then assume for contradiction $a_l\ne 0$ for some $l\ne j$. But then $\Ehor[F]$ must contain the horizontal node corresponding to the pivot  of $F_l$, which is not contained in $\Ehor[F_j]$. This gives a contradiction, and thus \rrefequiv implies \Mequiv.

For the other direction, assume that $\Ehor[F]$ is $M$-primitive and $e,e'\in \Ehor[F]$.
Denote $E_0$ the rref-equivalence class of $e$, and reorder the rref basis as $F_1,\dots, F_{u}, F_{u+1},\dots, F_m$ so that $\Ehor[F_j]\subseteq E_0$ for $j\leq u$ and $\Ehor[F_j] \cap E_0=\emptyset$ for $j>u$, and again write $F$ as $F=\sum_{l=1}^{m} a_lF_l$.

We claim that $\Ehor[\sum_{l=1}^{u} a_lF_l]\subseteq \Ehor[F]$. Note that by construction $\Ehor[\sum_{l=1}^{u} a_lF_l]\subseteq E_0$. If the claim were  false,  there would exist a node crossed by  $\sum_{l=1}^{u} a_lF_l$ but not by $F$. But then one of the  $F_j$ with $j>u$ has   to cross a node in $E_0$, which is impossible.

Since $\Ehor[F]$ is $M$-primitive we conclude that $\Ehor[F]=\Ehor[\sum_{l=1}^{u} a_lF_l]$ and by construction any pair of nodes in $\Ehor[\sum_{l=1}^{u} a_lF_l]$ is rref-related.
\end{proof}

We now have the tools to prove \Cref{thm:Mrltd}, which says that \Mrel horizontal vanishing cycles have proportional periods.  Before doing this, we illustrate what this result means in a simple example.

\begin{exa}[$M$-cross-related nodes]
  \label{exa:cross}
Suppose that we have a linear submanifold~$M$ in the genus $2$ stratum $\Omega\calM_{2,2} (1,1)$, and that  $p_0\in\pM$ is as shown in \Cref{fig:Mrltd}, with three horizontal nodes $e_1,e_2,e_3$, and no other nodes.   We suppose that one of the defining equations of~$M$ at~$p$ is
$$
  c_ 1\int_{\delta_1}\omega + c_2 \int_{\delta_2}\omega + c_3 \int_{\delta_3}\omega = 0,
$$
where the $c_i$ are non-zero complex numbers.  And suppose that there is no other defining equation that crosses a proper subset of the horizontal vanishing cycles $\{\lambda_{e_1},\lambda_{e_2},\lambda_{e_3}\}$.  Then the nodes $e_1,e_2,e_3$ are \Mrel, and \Cref{thm:Mrltd} implies that the periods of~$\omega$ over $\lambda_{e_i}$ are all proportional on~$M$.

There  in fact exist affine invariant manifolds that locally have the above description near certain boundary points.  For instance, one can take the eigenform locus $E_D$ (discovered, independently, by Calta~\cite{ca_veech} and McMullen~\cite{mc_bill}), for $D\equiv 0,1 \bmod 4$ a non-square positive integer.
\end{exa}

\begin{figure}[h]
\includegraphics[scale=0.75]{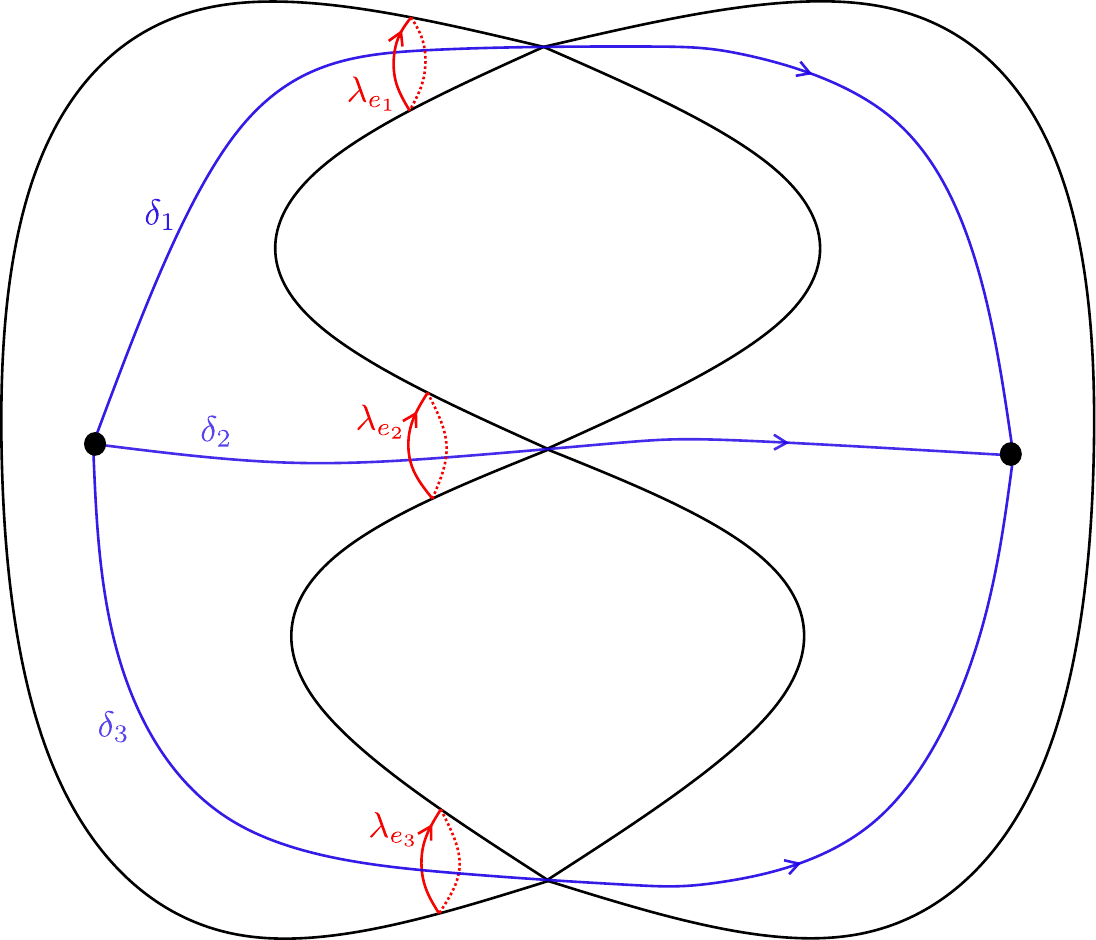}
\caption{A point in the boundary of $\Omega\calM_{2,2} (1,1)$, with three horizontal nodes.}
\label{fig:Mrltd}
\end{figure}
We are now ready to prove proportionality of periods over vanishing cycles for \Mrel nodes.
\begin{proof}[Proof of~\Cref{thm:Mrltd}]
The idea of the proof is to use \Cref{thm:divisorial} recursively to construct a suitable chain of divisorial undegenerations
\[
pt= \Gamma_0 \rightsquigarrow  \Gamma_1 \rightsquigarrow \cdots \rightsquigarrow \Gamma_d = \Gamma.
\]
As explained in \Cref{rem:chain}, many such chains can be constructed, and to specify a chain we need to specify at each step either a level transition that is smoothed or a horizontal node that is smoothed (in which case smoothing some other horizontal nodes may be required).

The main technical issue to deal with is choosing the appropriate chain of degenerations so that the monodromy argument can be applied. Recall that the results of~\cite{fred} apply to a rref basis of defining equations with respect to a chosen $\Gamma$-adapted homology basis. For an arbitrary chain of divisorial degenerations, it may happen that a homology basis that is $\Gamma$-adapted is not $\Gamma_j$-adapted for some $1\le j <d$. We will thus choose the chain of divisorial degenerations such that a homology basis can be chosen that is $\Gamma_j$-adapted simultaneously for all $0\le j\le d$. The chain of undegenerations that we choose is the following. Starting from $\Gamma_0=pt$, we first take degenerations that pinch horizontal nodes of top level in~$\Gamma$, i.e.~we take some $e_1\in \Ehori[0]$ and take the divisorial degeneration $\Gamma_0 \rightsquigarrow \Gamma_1$ that pinches $e_1$ --- which may possibly pinch some other horizontal nodes. If some top level horizontal node is not yet pinched, we take $e_2\in\Ehori[0]\setminus E(\Gamma_1)$, and take the divisorial degeneration $\Gamma_1\rightsquigarrow\Gamma_2$ that pinches $e_2$, and continue in such a way until $\Ehori[0]=E(\Gamma_k)$. We then let $\Gamma_{k+1}$ be the degeneration of $\Gamma_k$ that pinches the level transition between levels $0$ and $-1$, and then start pinching the horizontal nodes of level $-1$ in $\Gamma$ until all of $\Ehori[-1]$ is pinched, then pinch the level transition between levels $-1$ and $-2$, and so on.  We will want to track levels when we do this to eventually deal with a $\Gamma$-adapted basis.
We now choose a chain $pt= \Gamma_0 \rightsquigarrow  \Gamma_1 \rightsquigarrow \cdots \rightsquigarrow \Gamma_d = \Gamma$ of undegenerations satisfying all the conditions mentioned above and fix it for the rest of the proof.

The first issue is that a priori requiring some horizontal node to be pinched may lead to some other horizontal nodes at a different level beinf pinched. We prove that this is not the case. (In order to avoid confusion, in this proof the top level of a cycle always refers to the level with respect to $\lG$ and {\em not} with respect to some intermediate undegeneration).

{\bf Claim:} If $\Gamma_{i}\rightsquigarrow \Gamma_{i+1}$ is a purely horizontal divisorial degeneration appearing in the chain of degenerations $pt=\Gamma_0 \rightsquigarrow  \Gamma_1 \rightsquigarrow \cdots \rightsquigarrow \Gamma_d=\Gamma$ constructed above, then all horizontal nodes that are pinched in the degeneration $\Gamma_{i}\rightsquigarrow \Gamma_{i+1}$ have the same top level (in $\Gamma$).
\begin{proof}[Proof of claim]
Let $H(i,i+1)$ be the collection of horizontal nodes pinched in the degeneration $\Gamma_{i}\rightsquigarrow \Gamma_{i+1}$. By the discussion after the statement of \Cref{thm:divisorial}, we  know that for any pair of nodes  in $H(i,i+1)$ there exists a defining relation crossing exactly this pair of nodes. We then  apply the monodromy argument \Cref{prop:monodromy} to $\pMG[\lG_i]\subseteq \DG[\lG_i]$ and conclude that the vanishing cycles  of nodes in $H(i,i+1)$ are pairwise proportional on $\pMG[\lG_i]$.  By taking the limit of this proportionality relation to the boundary point $p_0\in \pMG$ we see that both horizontal nodes are of the same level in $\Gamma$, since otherwise the period over the vanishing cycle of one of the horizontal nodes would become zero which is impossible. This proves the claim.
\end{proof}
We now claim that there exists a $\Gamma$-adapted homology basis that is $\Gamma_j$-adapted for all $0\le j\le d$. Indeed, we first choose a $\Gamma$-adapted basis where the elements are ordered by level, with homology classes of higher top level appearing first. We then reorder the set of basis elements crossing horizontal nodes, so that they appear in the order in which the corresponding horizontal nodes are pinched in our chain of degenerations. That is, if $j< i$ then any cycle of the $\lG$-adapted basis crossing a node  in $\Ehor[\Gamma_j]$ is listed before any cycle crossing a node in $\Ehor[\Gamma_i]\setminus\Ehor[\Gamma_j]$. Note  that because of the claim proved above the reordering only changes the position of cycles of the same level so that cycles of higher levels still appear first. The resulting ordered $\lG$-adapted basis has the desired property that it is $\lG_j$-adapted for every $j$.
We then choose this ordered $\Gamma$-adapted homology basis, and take the rref basis of defining equations with respect to it, which will thus be a rref basis for every $\Gamma_j$ at once.

Since by \Cref{lm:Mrref}, \Mequiv and \rrefequiv are the same, it suffices to show that for any equation $F$ that is an element of the rref basis, the vanishing cycles of all horizontal nodes in $\Ehor[F]$ have proportional periods on~$M$.  Moreover, our proof will in fact give a certain rationality result: for any two nodes $e,e'\in\Ehor[F]$ there exist non-zero integers $n(e,e'), n'(e,e')$ such that
\begin{equation}\label{eq:constprop}
n(e,e')\lrra{F,\lambda_e}\int_{\lambda_e}\omega=n'(e,e')\lrra{F,\lambda_{e'}}\int_{\lambda_{e'}}\omega
\end{equation}
holds locally on~$M$ near~$p$. Let $H_i$ be the subset of $\Ehor[F]$ pinched at $\Gamma_i$, i.e. $H_i:=\Ehor[F]\cap E(\Gamma_i)$.

We will prove by induction on $i$ that the periods of~$\omega$ over all the vanishing cycles corresponding to nodes in $H_i$ are pairwise proportional on~$M$.  Since $H_d= \Ehor[F]$, the theorem will then follow.  For the base case, we take the smallest $j>0$ such that $H_j\ne\emptyset$.   Applying the monodromy argument,  ~\Cref{prop:monodromy}, to the equations lost at the degeneration $\Gamma_{j-1} \rightsquigarrow \Gamma_j$ shows then that for any two nodes in $H_j$, the periods over the corresponding vanishing cycles are proportional on $\pMG[\lG_{j-1}]$. Denote $G$ this proportionality, considered as one of the defining linear equations of $\pMG[\lG_{j-1}]$. Note that $G$ is completely contained in the bottom level, since it is a relation between vanishing cycles contained in the bottom level. Thus by \Cref{rem:lift} we can lift the linear equation $G$ for $\pMG[\lG_{j-1}]$ to a local defining equation $G'$  for $M$. A priori, $G'$ could have additional terms of lower levels that disappear when restricting to $\pMG[\lG_{j-1}]$. However, by construction of the chain of degenerations and by our choice of~$j$, $\topl(F)=\topl(G)=-L(\Gamma_{j-1})$ is equal to the bottom level of the graph $\lG_{j-1}$. Thus there are simply no levels below that level, and thus $G'$ cannot have any additional summands at lower level. Thus it follows that $G'$ is  a proportionality on $M$ of these two nodes in~$H_j$, which is thus one of the defining equations of~$M$. This concludes the proof of the base case of the induction.

For the inductive step, we will prove the proportionality of periods of $\omega$ over the vanishing cycles for any two nodes in $H_{i+1}$, assuming that this holds for any pair of nodes in $H_i$. This statement is vacuously true if $H_i=H_{i+1}$, so the inductive step holds automatically unless $i\ge 1$ and $\Gamma_i\rightsquigarrow \Gamma_{i+1}$ is a purely horizontal degeneration. For such an undegeneration, denote $\lbrace h_1,\dots,h_k\rbrace:=E(\Gamma_{i+1})\setminus E(\Gamma_{i})$ the new nodes that are pinched, all of which are horizontal. Since the degeneration $\Gamma_i \rightsquigarrow \Gamma_{i+1}$ is divisorial, the number of defining equations of~$M$ lost on $\Gamma_{i+1}$ in addition to those lost on~$\Gamma_{i}$ must be equal to $k-1$, one less than the number of horizontal nodes pinched.  Each of the lost equations has a pivot variable in the \rref, and by \Cref{thm:divisorial} there exists exactly one horizontal node among the $h_1,\ldots,h_k$ (which by renumbering the $h_j$ we will assume to be $h_k$) that is not a pivot for any of the lost equations. Then $h_1,\ldots,h_{k-1}$ all correspond to pivots of the lost equations, and since all the equations are in \rref, we see that $F$ cannot cross any of these, since~$F$ was lost at the base case degeneration.  On the other hand, $H_i\subsetneq H_{i+1}$ means that $\Ehor[F]$ must contain some $h_j$, and thus we must have $h_k\in\Ehor[F]$, and $H_{i+1}=H_i\sqcup \lbrace h_k\rbrace$.

We now apply the monodromy argument to $F$ with respect to the boundary stratum $\DG[\Gamma_{i+1}]$ considered within the closure boundary stratum $\overline{\DG[\Gamma_{j-1}]}$, where $j$ was the index introduced in the base case. This is to say, we consider monodromy around the vanishing cycles for the nodes in $E(\Gamma_{i+1})\setminus E(\Gamma_{j-1})$, obtaining in this way a defining equation $G$ of $\pMG[\lG_{j-1}]$ that is a linear combination of periods over the vanishing cycles in $H_i \sqcup \{h_k\}$ and possibly also periods over some vertical vanishing cycles crossed by $F$, with all the coefficients non-zero. Note that all these vanishing cycles are of level $-L(\Gamma_{j-1})$ when considered in $\Gamma_{j-1}$ and thus the equation $G$ can be lifted to a linear equation on $M$; which by abuse of notation we will denote $G$ again.
First consider the case where all the degenerations $\Gamma_{j-1}\rightsquigarrow \cdots\rightsquigarrow\Gamma_{i+1}$ are purely horizontal. In this case $\Gamma_{i+1}$ has the same number of levels as $\Gamma_{j-1}$ and $F$ does not cross any vertical vanishing cycles.
By induction, we already know that the periods of~$\omega$ over the vanishing cycles of nodes in $H_i$ are pairwise proportional on~$M$, and substituting this into~$G$ implies that the period of $\omega$ over the remaining vanishing period $\van[h_k]$  is also proportional to these, on $\pMG[\lG_i]$. The coefficient of $\int_{\van[h_k]}\omega$ in~$G$ is given by monodromy, and is thus non-zero, so the period $\int_{\van[h_k]}\omega$ is a non-zero multiple of the period over the vanishing cycle for any node in $H_i$. Proceeding exactly as in the base case of induction, we conclude that this proportionality is actually satisfied on~$M$, and not only on $\pMG[\lG_i]$. Tracing through the argument, we see that the proportionality relations are of the form \eqref{eq:constprop}.

It remains to treat the case where some undegeneration in the chain $\Gamma_{j-1}\rightsquigarrow \cdots\rightsquigarrow\Gamma_{i+1}$ is vertical. In this case the application of the monodromy argument to $F$ might pick up additional contributions from the vertical vanishing cycles that $F$ crosses.
Note that the contributions from vertical vanishing cycles are (up to multiplication by a constant) independent of the order in which undegenerations are performed and which boundary point is chosen to apply the monodromy argument; this follows from  the computation leading to \eqref{eq:vereq} and is incorrect for horizontal vanishing cycles.  Thus to see
to see that  contribution from vertical vanishing cycles vanishes we can, starting from $\Gamma_{j-1}$, only perform vertical degenerations and then the contributions vanish by \eqref{eq:vereq}.
\end{proof}
\begin{cor}\label{cor:Mrltdlvl}
Any pair of \Mrel nodes has the same level.
\end{cor}
\begin{proof}
Let $e,e'\in \Ehor$ be a pair of \Mrel nodes, so that by~\Cref{thm:Mrltd} the periods over the corresponding vanishing cycles are proportional on~$M$. The rescaled limits of these periods are the (non-zero) residues of the twisted differential at the corresponding nodes. If one node is lower than the other, by definition of a multi-scale differential compatible with a level graph this means that the limit of the ratio of these residues must be equal to zero. Since the residues are proportional with a constant coefficient, this means that both residues must be identically zero, which is impossible, as the multi-scale differential must have simple poles at all horizontal nodes by definition.
\end{proof}
We now investigate defining equations that are  \nhces. We recall that an equation $F$ is called \nhc if it does not cross any horizontal nodes of level $\topl(F)$. This allows the possibility that $F$ might cross horizontal nodes of levels below $\topl(F)$.
\begin{lm} \label{lm:non-horiz}
Let $F$ be a defining equation of~$M$ that does not cross any horizontal nodes at level $\topl(F)$, but crosses some horizontal node $e\in\Ehori$ at level $i<\topl(F)$.  Then $F$ can be written as the sum $F=H+G$ of defining equations such that $\topl(G)=\topl(F)>\topl(H)$, where $G$ crosses no horizontal nodes, at any level, and $H$ is \thc with $\topl(H)$ being the maximal level of a horizontal node crossed by $F$.
\end{lm}
\begin{proof}
Using \Cref{rem:chain}, we construct the following chain of (divisorial except for the last two) undegenerations:
\[
  pt=\lG'_0\rightsquigarrow\dots \rightsquigarrow\lG'_k\rightsquigarrow\lG'\rightsquigarrow\lG
\]
Here each $\lG'_i\rightsquigarrow\lG'_{i+1}$ is a divisorial degeneration pinching some horizontal node in $\Ehor[F]$, and we perform such divisorial degenerations until all nodes in $\Ehor[F]$ are pinched, i.e. $E(\lG'_k)=\Ehor$. Then $\Gamma_k'\rightsquigarrow\Gamma'$ is the purely vertical degeneration that closes the level passage between $\topl(F)$ and $\topl(F)-1$. Finally, $\lG'\rightsquigarrow\lG$ is the remaining degeneration, that closes all other level passages and all other horizontal nodes of~$\lG$.

Then~$F$ is \thc for all $\lG'_j$ but not for $\lG'$. Since every defining equation for $\pMG[\lG']$ is induced by an equation of $\pMG[\lG'_k]$, and every defining equation of $\pMG[\lG'_{j+1}]$ is induced by an equation of $\pMG[\lG'_j]$, it follows that each defining equation of $\pMG[\lG']$ is induced from a defining equation of~$M$ at~$p$ that does not cross {\em any} horizontal nodes. Thus we can find an equation $G_0$ with the same top level restriction as $F$, i.e. $(G_0)_{\topl}=F_{\topl}$, but such that $G_0$ crosses {\em no} horizontal nodes. In particular then $\topl(F-G_0)<\topl(F)$. Now either $F-G_0$ is \thc, or we can proceed inductively and find $G$ as desired.
\end{proof}

We can now prove the decomposition of the linear equations.
\begin{proof}[Proof of~\Cref{thm:decomp}]
We proceed by induction on $\#\Ehor[F]+\topl(F)$. If $F$ crosses no horizontal nodes, we set $G:=F$ and are done. Otherwise, if all nodes in $\Ehor[F]$ are at  levels strictly below $\topl(F)$, we write $F=H+G$ as provided by \Cref{lm:non-horiz}, and apply the induction hypothesis on $H$.

We are thus left with the case that $F$ crosses some horizontal node $e\in \Ehor[F]\cap\Ehori[\topl(F)]$.  Given any defining equation of~$M$, by the definition of primitivity, there exists some primitive equation $P$ that crosses a subset of the horizontal nodes crossed by the original equation, i.e. $\Ehor[P]\subseteq \Ehor[F]$ and $\topl(P) \ge \topl(F)$. By \Cref{lm:non-horiz}, we can further assume that $\topl(P)=\topl(F)$.  Then there exists a constant $c\in\CC^*$ such that either $F=c P$ or $\Ehor[F-cP]\subsetneq \Ehor[F]$. We can then apply the induction to $F-c P$, and since $\Ehor[P],\Ehor[F-c P]\subseteq \Ehor[F]$, condition $(2)$ of the statement of the Theorem will be satisfied.
\end{proof}

\smallskip
While all the above statements were for arbitrary defining equations, for the rref basis $F_1,\dots,F_m$ we can obtain more precise results, determining the coefficients of the equations explicitly. While this, our most precise, result, is more technical, it will be crucial in enabling us to compute the analytic equations of~$M$ in plumbing coordinates in~\Cref{sec:transverse}, and in particular prove~\Cref{thm:transverse}.

\begin{prop} \label{prop:eqsquantitative}
Let $F_1,\dots,F_m$ be the rref basis, written as in~\eqref{eq:Fformsimple}. Then
\begin{enumerate}
\item each $F_l$ does not cross any horizontal node at level below $\topl(F_l)$.
\item If $F_l$ crosses $e,e'\in\Ehori[\topl(F)]$, then there exist two non-zero integers $n_1,n_2\in\ZZ$ such that the equation
\[n_1\lrra{F_l,\lambda_e}\int_{\van[e]}\omega=n_2\lrra{F_l,\lambda_{e'}}\int_{\van[e']}\omega\] holds on~$M$ in a neighborhood of $p$.
\item For any level $i$, the equation \[
R_i(F_l)=\sum_{\substack{e\,:\,\ltop\ge i>\lbot}}m_{e,i}\langle F_l,\lambda_e\rangle\int_{\lambda_{e}}\omega =0
\] holds on $M$ in a neighborhood of $p$, where we recall $m_{e,i}$ defined in \eqref{eq:se}.
\end{enumerate}
\end{prop}
\begin{proof}
We first prove $(1)$. If $F_l$ crosses any horizontal nodes of level below $\topl(F)$, then consider the decomposition $F_l=H_1+\dots+ H_k+G$ provided by \Cref{thm:decomp}. After reordering the $H_i$ we can assume that $\topl(H_1)<\topl(F)$. Writing $H_1=\sum_{k}a_kF_k$, it follows then that $\topl(F_k)<\topl(F_l)$ whenever $a_k\neq 0$. Furthermore, there must exist $j$ such that $a_j\neq 0$ and $F_j$ crosses some horizontal nodes. Let $e$ be the horizontal node corresponding to the pivot of $F_j$. Then $e\in \Ehor[H_1]\subseteq \Ehor[F_l]$, which is a contradiction since the pivot node can only appear in $F_j$, and in no other equation of the rref basis.

The proof of $(2)$ was the content of \eqref{eq:constprop}, while the statement $(3)$ was proved already in the  proof of \Cref{thm:vertres}.
\end{proof}

\section{Equations near the boundary in plumbing coordinates}\label{sec:transverse}
Using the restrictions on linear equations obtained in \Cref{prop:eqsquantitative}, we can now convert the linear equations in period coordinates into analytic equations in plumbing coordinates and thus prove~\Cref{thm:localequations}. The most precise technical result that we prove in this direction is \Cref{prop:locirred}.

\subsection{Converting equations from period to plumbing coordinates}

While periods of the differential are not globally well-defined on~$U$, recall that in~\cite{fred} the so-called \emph{log periods} were defined (these are related to perturbed periods of \cite{BCGGMmsds}, and to the expressions for periods used in \cite[Lemma 3.8]{ben}). These are well-defined analytic functions on~$U$, obtained by subtracting logarithmic terms, as we now recall. As always, we work in the neighborhood $U$ of $p_0\in\pM$, and consider defining equations of~$M$ at a smooth point $p=(X,\omega)\in M\cap U$; now we will also fix a class $\gamma\in \relhomZ$. Recall that coordinates on $U$ are given by $b:=(\eta,\underline{t},\underline{h})$ where $\eta\in\DG$ can be thought of as a twisted differential, and thus local coordinates for $\eta$ are given by its periods, $\underline{t}=\lbrace t_{-1},\dots,t_{-L(\lG)}\rbrace$ are the level scaling parameters, and $\underline{h}=\lbrace h_e\rbrace_{e\in\Ehor}$ are the plumbing parameters at horizontal nodes. The {\em log period} of $\omega$ along $\gamma$ is defined as
\[
\logP(\omega):= \dfrac{1}{\scl[{\topl(\gamma)}]}\left[ \int_{\gamma} \omega - \sum_{e\in E} \lrra{\gamma,\lambda_e} r_e(\omega)\ln(s_e)\right ]\,,
\]
where $r_e(\omega):= \dfrac{1}{2\pi i}\int_{\lambda_e}\omega$ and, as in~\cite{BCGGMmsds}, we denote
\[
\scl[{i}]:= \prod_{k=-i}^{-1} t_i^{a_i},
\]
with the $a_i$ defined by \eqref{eq:se}. Here $\gamma$ is extended smoothly to nearby curves using the Gauss-Manin connection.  A priori, this might not be well-defined because of the monodromy of the Gauss-Manin connection, but the logarithmic terms are chosen exactly to cancel out this monodromy, which yields

\begin{prop}[{\cite[Thm. 5.2]{fred}}]\label{prop:logp}
The log period $\logP$ is a well-defined analytic function on $U$. Furthermore, if $\gamma$ is non-horizontal, then
\[
\logP(b)= \int_{\gamma}\eta+ H(b)\,,
\]
where $H$ is an analytic function on $U$ that vanishes identically on $\DG\cap U\subset U$.
\end{prop}

We will now rewrite the  defining equations of~$M$ at~$p$ in terms of log periods, and then express them in plumbing coordinates.
We briefly recall our setup for writing in linear equations.

 Let $F_1,\ldots, F_m$ be the rref basis of defining equations for $M$,  with respect to a fixed $\lG$-adapted basis.

To lighten the notation we focus on one of the equations $F:=F_k$ for now.

Then we can write
 \[
  F(X,\omega)=\sum_{i=-L(\lG)}^{\topl(F)}\left(\sum_{l=1}^{\lgcr}A_l^{(i)}\int_{\crC_l}\omega
 +\sum_{l=1}^{\lghf}B_l^{(i)}\int_{\hfC_l}\omega\right)
 \]
as in ~\eqref{eq:Fform}. Here we recall that $\crC$ are the  \hccs of level $i$ and $\hfC$ are the \nhccs  of level $i$ for a fixed $\lG$-adapted basis.

By \Cref{prop:eqsquantitative}(1) all coefficients $A_l^{(j)}$ in~\eqref{eq:Fform} are zero for $j<\topl(F)$. We thus denote $i:=\topl(F)$ from now on, and denote $A_l:=A_l^{(i)}$. In terms of log periods, we compute
\[
\begin{split}
F=&\sum_{l=1}^{c(i)} A_{l}\left(\scl[i] \logP[\delta_l^{(i)}]+\sum_{e\in E}\lrra{\delta_l^{(i)}, \lambda_e}\cdot\int_{\lambda_e}\omega \cdot\ln(s_e)\right) \\
&+\sum_{j=-L(\lG)}^{i}\sum_{l=1}^{d(j)}B^{(j)}_l \left(\scl[j] \logP[\alpha_l^{(j)}]+\sum_{e\in E}\lrra{\alpha_l^{(j)}, \lambda_e}\cdot\int_{\lambda_e}\omega \cdot\ln(s_e)\right)\\
=&
\sum_{l=1}^{c(i)} A_{l}\scl[i] \logP[\delta_l^{(i)}]+\sum_{j=-L(\lG)}^{i}\sum_{l=1}^{d(j)}B^{(j)}_l \scl[j] \logP[\alpha_l^{(j)}]+ \sum_{l=1}^{c(i)} A_{l}\cdot\int_{\lambda_l^{(i)}}\omega \cdot\ln(s_l^{(i)})
\end{split}
\]
where we used \Cref{prop:eqsquantitative} (3) to substitute $R_j(F)=0$ to obtain the cancellation of terms for the second equality. For future use, denote
\[
L(b):=\sum_{l=1}^{c(i)} A_{l}\scl[i] \logP[\delta_l^{(i)}](b)+\sum_{j=-L(\lG)}^{i}\sum_{l=1}^{d(j)}B^{(j)}_l \scl[j] \logP[\alpha_l^{(j)}](b).
\]
If $F$ is a \nhce then all coefficients $A_l$ are zero and we define
\begin{equation}\label{eq:G}
G(b):= \dfrac{1}{\scl[\topl(F)]}L(b)= \dfrac{1}{\scl[\topl(F)]}F(b)
\end{equation}
which is then a holomorphic function on~$U$.

If $F$ is a \hcc, then by~\Cref{thm:Mrltd}, the periods over vanishing cycles for all nodes in $\Ehor[F]$ are proportional. Since $F$ is an element of the rref basis, we have $A_1 =1$. Since the coefficients of proportionality determined explicitly in~\eqref{eq:constprop} are rational, it follows that there exist numbers $q_l\in{\mathbb Q}\setminus\lbrace 0\rbrace$ such that
\[
A_l\cdot\int_{\lambda_l^{(i)}} \omega= q_l \cdot\int_{\lambda_1^{(i)}}\omega.
\]
We can thus write
\[
F(b)=L(b)+  \left(\int_{\lambda_1^{(i)}}\omega \right) \cdot \left(\ln(s_1^{(i)})+\sum_{l=2}^{c(i)}  q_{l} \ln(s_l^{(i)})\right)
\]
By clearing denominators (this is where the rationality of~$q_l$ is used), the equation $F(b)=0$ is then equivalent to
\[
  \dfrac{n_1}{\int_{\lambda_1^{(i)}}\omega}L(b)+  \sum_{l=1}^{c(i)}  n_l\ln(s_l^{(i)}) = 0
\]
for some non-zero integers $n_l$, and without loss of generality we can assume $\gcd(n_l)=1$. By exponentiating, this is in turn equivalent to
\begin{equation}\label{eq:toric}
e^{f(b)}\prod_{l=1}^{c(i)} (s_l^{(i)})^{n_l}=1
\end{equation}
where we set
\[
f(b):= \dfrac{n_1}{\int_{\lambda_1^{(i)}}\omega}L.
\]
Since the point $p_0$ for which all $s_l^{(i)}=0$ is contained in $\pM$, we must have $f(0,0,0)=0$. Thus it cannot happen that all $n_i$ have the same sign. By separating terms with $n_l$ positive and negative, we can rewrite~\eqref{eq:toric} as
\begin{equation}\label{eq:H}
0=H(b):= e^{f(b)}s^{I}-s^{J}
\end{equation}
for some monomials $s^{I}$ and $s^{J}$ in the plumbing parameters $s_l^{(i)}$. We have now converted an element~$F$ of the rref basis for defining equations of~$M$ at~$p$ to plumbing coordinates and can now repeat the same process for all remaining equations in the rref basis. Before proceeding with the general setup, we give an example of how this works in practice.
\begin{exa}
Consider a boundary point $p_0$ such that the corresponding stable curve is irreducible --- and in particular the level graph has one vertex and some horizontal edges. Let $\lambda_1,\lambda_2$ be two horizontal vanishing cycles, and let $\delta_1,\delta_2$ be crossing curves for them, as shown in \Cref{fig:eqns}.
  \begin{figure}
\includegraphics[scale=0.75]{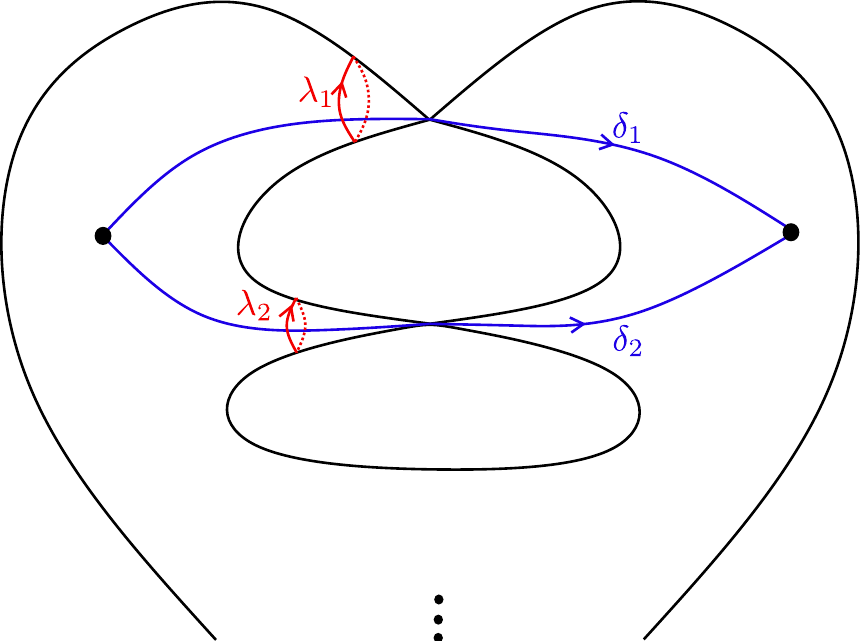}
\caption{A boundary point with two horizontal node vanishing cycles and two curves crossing these}
\label{fig:eqns}
\end{figure}
Suppose that $M$ is locally cut out by the two equations in period coordinates given in the table below, where we note that~\Cref{thm:Mrltd} applied to the first equation implies an equation relating the two periods in the second --- so that our machinery so far does not automatically produce any further equations. We will now demonstrate the procedure to convert these into equations in plumbing coordinates $b=(\eta,\underline{h})$ (since $\Gamma$ has only one level, there are no scaling parameters $\underline{t}$). The result is given in the table below
\renewcommand{\arraystretch}{1.7}
\begin{table}[h]
  \begin{tabular}{|l|l|} \hline
      {\bf Period equations} & {\bf Plumbing equations}  \\ \hline
$\int_{\delta_1}\omega - \int_{\delta_2}\omega = 0$ & $e^f s_1 - s_2 = 0$ \\ \hline
      $\int_{\lambda_1}\omega - \int_{\lambda_2} \omega = 0$  & $\int_{\lambda_1}\omega - \int_{\lambda_2} \omega = 0$    \\ \hline
    \end{tabular}
\end{table}

To convert the first period equation, we first express the first equation in terms of log periods:
\begin{equation*}
    \left(\logP[\delta_1] + \int_{\lambda_1} \omega \cdot \ln s_1\right) - \left( \logP[\delta_2] + \int_{\lambda_2} \omega \cdot \ln s_2\right) = 0.
\end{equation*}
Substituting here the second period equation yields
\begin{equation*}
    \logP[\delta_1]-\logP[\delta_2] + \left(\int_{\lambda_1}\omega\right)\left(\ln s_1 -\ln s_2\right) = 0.
\end{equation*}
Next we divide through and set $L= \logP[\delta_1]-\logP[\delta_2]$, which gives
\begin{equation*}
  \frac{L}{\int_{\lambda_1} \omega} + (\ln s_1 -\ln s_2)= 0,
\end{equation*}
Exponentiating and setting $f=\frac{L}{\int_{\lambda_1} \omega}$, we finally arrive at
\begin{equation*}
  e^f s_1 s_2^{-1} = 1,
\end{equation*}
so the first equation becomes $e^fs_1 - s_2 =0$ in plumbing coordinates, as claimed.

The second period equation extends holomorphically to the boundary. Indeed, it does not involve any \thc cycles, and thus no manipulations are necessary.
\end{exa}

\subsection{Rearranging equations.}
Going back to the general case, suppose that among the rref basis $F_1,\dots, F_m$, there are $u$ equations that are non-horizontal, and $u'=m-u$ that are \hces. We denote by $G_1,\dots,G_u$ the results of converting the \nhces to plumbing coordinates according to \eqref{eq:G}, and denote by $H_1,\dots,H_{u'}$ the results of converting the \hces to plumbing coordinates according to \eqref{eq:H}. These can be then written as
\begin{equation}\label{eq:convert}
\begin{split}
G_k(b)&=\dfrac{1}{\scl[\topl(F_k)]} L_k(b),\\
H_k(b)&= e^{f_k(b)}s^{I_k}-s^{J_k}
\end{split}
\end{equation}
for some monomials $s^{I_k},s^{J_k}$ in the variables $s_{l}^{(\topl(F_k))}$. Note that as $\underline{t}$ and $\underline{h}$ tend to zero, the equations $G_k$ tend to the defining linear equations for $\pMG$.

If we now define
\[
  V:= \{b\in U\,:\, G_1(b)=\dots=G_u(b)=H_1(b)=\dots=H_{u'}(b)=0\},
\]
then the defining linear equations of~$M$ can be rewritten analytically in plumbing coordinates as the equations defining~$V$ in plumbing coordinates.

We have thus proven
\begin{prop}\label{prop:locirred}
The local irreducible component~$\overline{Z}$ of $\oM$ at $p_0$ containing $p$ is a local irreducible component of $V$.
\end{prop}
We have thus proven a big part of \Cref{thm:localequations}: we have converted the defining equations into plumbing coordinates, and have given their explicit form in~\eqref{eq:convert}.  The rest of the proof is a matter of organizing the equations.
\begin{proof}[Proof of \Cref{thm:localequations}]

We now rearrange the equations to reveal some of the underlying structure. Let $l(1),\dots,l(u')$ be the pivots of those equations $F_{j_1},\dots,F_{j_{u'}}$ that are \thc. After a change of coordinates
\[
x_{l}^{(i)}:=\begin{cases} e^{f_{j_k}(b)}s_{l(k)}^{(\topl(F_{j_k}))} & \text{ if } l=l(k), i=\topl(F_{j_k}),\\
s_l^{(i)} & \text{ otherwise},
\end{cases}
\]
the equations $H_k$ take the form
\begin{equation}
  \label{eq:binomial}
  H_k= x^{I_k}-x^{J_k}
\end{equation}
where $I_k,\, J_k$ are the monomials from \eqref{eq:convert}.

After this change of coordinates, we write the coordinates on $U$ as $(\underline{y},\underline{x})$ where $\underline{y}$ are all coordinates not involving horizontal nodes, and $\underline{x}$ is the set of plumbing parameters at the horizontal nodes that we just defined. Note that $\underline{y}$ can be separated further into the rescaling parameters $\underline{t}$ and  the periods $\int_{\alpha_l^{(j)}} \eta$. We furthermore separate $\underline{x}$ into sets of coordinates corresponding to individual \Mequiv classes, writing $\underline{x}=(\underline{x}_1,\ldots,\underline{x}_a)$. In these coordinates  the  local irreducible component $\overline{Z}$ of $\oM\cap U$ containing $p$ is an irreducible component of the product
\begin{equation}\label{eq:prod}
V= \{\underline{y}: G_1(\underline{y})=\dots =G_u(\underline{y})=0\} \times \mathop{\prod}_{l=1}^{a} \lbrace \underline{x}_l : \underline{H}_l(\underline{x}_l)=(0,\ldots,0)\rbrace
\end{equation}
where each $\underline{H}_l$ is the vector of all equations $H_k$ crossing nodes in the \Mequiv class~$\underline{x}_l$; this is possible since all nodes crossed by $H_k$ lie in the same \Mequiv class.

We now show that locally in the analytic topology $\overline{Z}$ is isomorphic to a product of~$\CC^n$ and varieties defined by binomial equations.
Since in the coordinates given by $(\underline{y},\underline{x})$ each equation $H_k$ is a difference of two monomials, it remains to show that the factor $ \{\underline{y}: G_1(\underline{y})=\dots =G_u(\underline{y})=0\} $ is smooth and thus locally isomorphic to $\CC^n$. This follows in particular from the proof of \Cref{cor:smooth}.

To finish the proof of the Theorem, we recall the relation between binomial equations and toric varieties. Recall that by definition a toric variety $X$ contains an algebraic torus $(\CC^*)^n$ as an open dense subset, so that the action of $(\CC^*)^n$ extends to $X$ (note that here we do not require $X$ to be a normal variety). By \cite[Lemma 1.1]{SturmfelsI} (see also \cite{SturmfelsII}) the zero locus of a binomial prime ideal in $\CC[x_1,\ldots,x_n]$ is an irreducible toric variety. The ideal generated by the equations for $V$ is generated by binomials but in general is not a prime ideal. Using the special form \eqref{eq:convert} of the equations for $V$ we will explicitly construct an embedding of $(\CC^*)^n$ in $Z$ and thus show that $\overline{Z}$ is locally a toric variety.

Since $V$, and hence also $Z$, is a product  in $(\underline{y},\underline{x})$-coordinates, it suffices to define the algebraic torus on each component cut out by $\underline{H}_l$. The crucial observation is that for each equation $H_l$ the pivot variable only appears in $H_l$, and in no other equations. Thus we define the $(\CC^*)^n$-action explicitly in coordinates $z_1,\ldots,z_c$, where $c$ is the number of non-pivots for $\underline{H}_l$, by sending $z_k$ to $x_k$ if $x_k$ corresponds to a non-pivot, and for a pivot variable we define $x_k$ as a function of $z_1,\ldots,z_c$ by solving the equation $H_k$ for $x_k$.
\end{proof}

\begin{figure}
\includegraphics[scale=0.6]{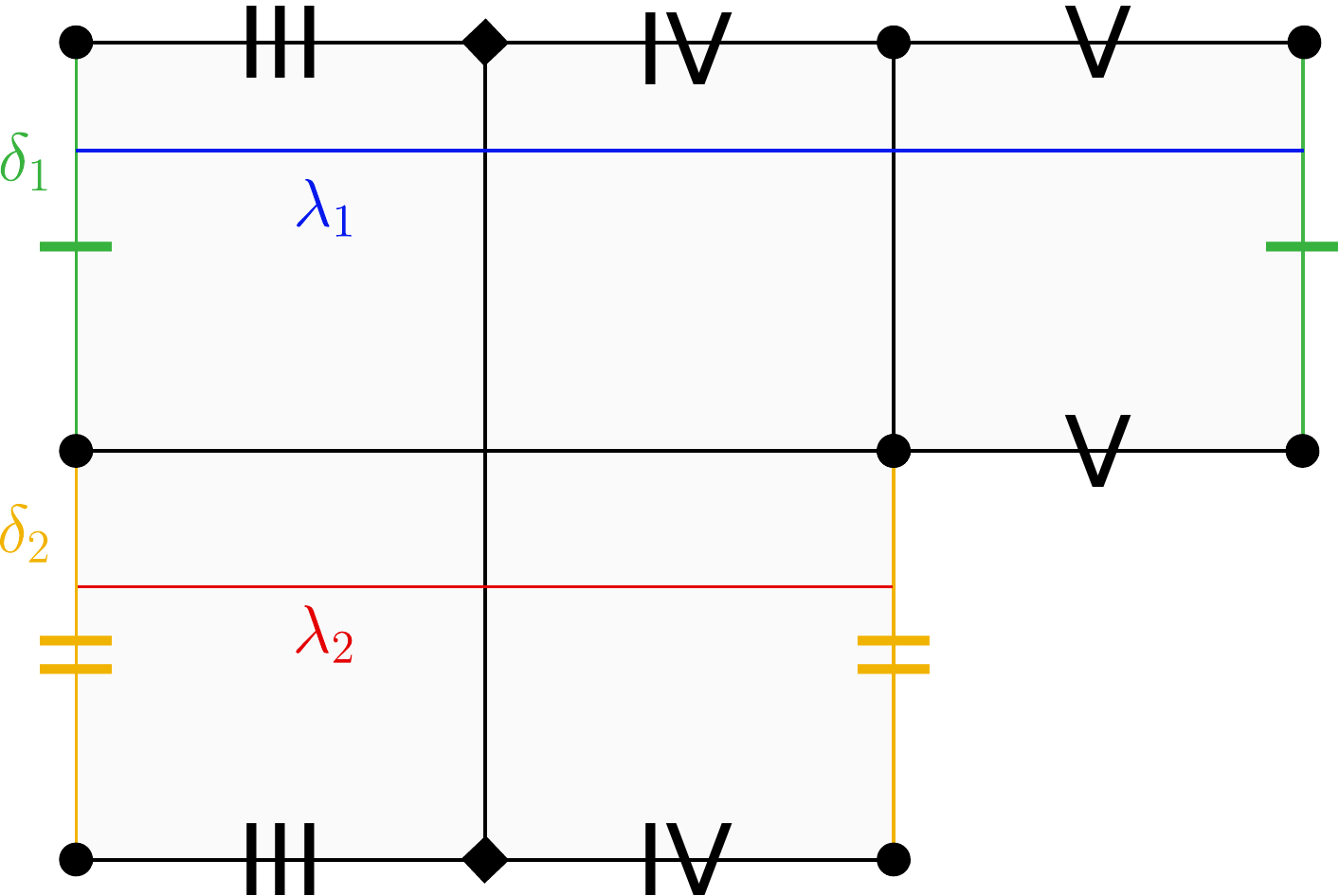}
\caption{The square tiled surface from \Cref{ex:squaretiled}}
\label{fig:squaretiled}
\end{figure}
\begin{exax}\label{ex:squaretiled}
The following example shows that the local irreducible component~$\overline{Z}$ of the linear subvariety may not be normal, already for Teichm\"uller curves. Indeed, every $2$-dimensional affine invariant submanifold contains {\em completely periodic surfaces} $(X,\omega)$, by which we mean that $(X,\omega)$ is a union of horizontal cylinders. Furthermore, for such $(X,\omega)$ the moduli of all horizontal cylinders on it are pairwise commensurable, by the Veech dichotomy (and also as easily follows from \Cref{thm:divisorial}). Furthermore, all core curves of horizontal cylinders are pairwise proportional on $M$ and  there exists a choice of cross curves of the horizontal cylinders such that all cross curves are pairwise proportional as well. For two horizontal cylinders~$C_1$ and~$C_2$ on~$X$ let~$e_1$ and~$e_2$ be the resulting horizontal nodes on the nodal curve obtained by applying $\begin{pmatrix} e^t&0\\0&e^{-t}\end{pmatrix}$ to the horizontal cylinders, while fixing the rest of the surface. Then we can convert a linear relation $\int_{\delta_1}\omega=c\int_{\delta_2}\omega$ among the periods of core curves into the analytic equation
\[
x^a=y^b
\]
where we denote $\dfrac{a}{b}:=\dfrac{m(C_1)}{m(C_2)}$, with $\gcd(a,b)=1$, the ratio of the moduli of the two cylinders. For example we consider the square-tiled surface from \Cref{fig:squaretiled} in the stratum $\Omega\calM_{2,1}(2)$ with a horizontal cylinder of modulus $3$ stacked on top of a horizontal cylinder of modulus $2$.
In this case the two linear relations defining $M$ are
\[
\int_{\delta_1}\omega=\int_{\delta_2}\omega;\quad 2\int_{\lambda_1}\omega=3\int_{\lambda_2}\omega,
\]
Converting them into analytic equations, we see that the local irreducible component $\overline{Z}$ is isomorphic to
\[
\CC\times \{(x,y)\,|\, x^2=y^3\}.
\]
which is a product of $\CC$ and a cusp, which is not normal, for example since the singular locus has codimension $1$.
\end{exax}

\begin{cor}\label{cor:smooth}
If all defining equations of $M$ are non-horizontal, then each local irreducible component $\overline{Z}$ of $\oM$ is smooth and transverse to the vertical boundary stratum given by $\{t_{-1}=\dots=t_{-L(\lG)}=0\}$.
\end{cor}
\begin{proof}
Indeed, in this case of only non-horizontal equations we only have the first factor present in~\eqref{eq:prod}, and thus
\[
V=  \{\underline{y}:G_1(\underline{y})=\dots =G_u(\underline{y})=0\}
\]
For each equation $G_k$ the pivot corresponds to a period $\int_{\alpha_{l(k)}^{(\topl(G_k))}} \eta$, since $G_k$ is non-horizontal.
By \Cref{prop:logp}, the Jacobian of the set of equations $G_1,\dots,G_u$ with respect to coordinates $\underline{y}$ is in \rref and has the same pivots as the original linear equations $F_{j_1},\dots,F_{j_u}$ corresponding to non-horizontal equations.
In particular $V$ is smooth and irreducible, and since it contains $\overline{Z}$, it must coincide with $\overline{Z}$. Furthermore, the normal space to $V$ is generated by $\int_{\alpha_{l(k)}^{(\topl(G_k))}} \eta$ for $k=1,\ldots,u$  and thus we can choose $t_{-1},\dots,t_{-l(\lG)}$ as part of a local coordinate system on $V$, which shows that $Z$ is transverse to $\{t_{-1}=\dots=t_{-L(\lG)}=0\}$.
\end{proof}

The condition of this Corollary is satisfied for example if $\oM$ is disjoint from the closed boundary divisor of $\LMS$ that corresponds to graphs that have a horizontal edge. In \Cref{sec:HW} we will apply this Corollary to obtain a compactification of Hurwitz spaces. We are now ready to prove our result about smoothing a collection of nodes of $\lG$.

\begin{proof}[Proof of~\Cref{thm:transverse}]
First note that the variety defined by equations $G_1(\underline{y})=\dots =G_u(\underline{y})=0$ is smooth and irreducible; we denote it $Y$.
As the local irreducible component $\overline{Z}$ of $M$ is an irreducible component of $V$, which is a direct product, it follows that $\overline{Z}$ is a product of irreducible components of the factors, and we write is as $\overline{Z}=Y\times\prod_{l=1}^a X_l$, $X_l$ denote the individual factors, which are given by equations in variables $\underline{x}_l$. As in the proof of~\Cref{cor:smooth} above, we know that there is a local coordinate system on $Y$ including $(t_{-1},\dots,t_{-L(\lG)})$. Thus for any sufficiently small collection of $t_i$, there exists a point in $Y$ with these $\underline{t}$-coordinates, which is to say that any collection of level passages in $\Gamma$ can be smoothed, while remaining in $\oM$.

To show that any \Mequiv class of horizontal nodes can be smoothed while remaining in~$\oM$, we simply observe that since $\overline{Z}$ is a product, and contains the flat surface $p\in\overline{Z}\cap M$, it means that the coordinates $\underline{x}_l(p)$ of this point are all non-zero, while $\underline{x}_l(p)\in X_l$. But then the point with all the same $\underline{y}$ and $\underline{x}$ coordinates as $p_0$, except with  coordinates $\underline{x}_l(p)$, lies in the product $\overline{Z}=Y\times\prod_{l=1}^a X_l$, which is exactly to say the $l$-th \Mequiv class of nodes has been smoothed.
\end{proof}

\subsection{Application: a smooth compactification of Hurwitz spaces}\label{sec:HW}
As an application of our study of the local analytic equations of linear subvarieties, we construct a smooth compactification of Hurwitz spaces.

Recall that Hurwitz spaces are moduli spaces of rational functions on Riemann surfaces with prescribed ramification multiplicities. By associating to a rational function $f:X\to\PP^1$ its exact differential $df$ we can consider Hurwitz spaces as subvarieties of meromorphic strata. Being an exact differential is characterized by the vanishing of all absolute periods, which are linear conditions in period coordinates. We can thus realize Hurwitz spaces as linear subvarieties of strata. The Hurwitz spaces we consider here are a ``rigidified" version of the standard Hurwitz spaces where we mark all points lying over a branch point. If we only mark the points over two fibers, for example the fiber over $0$ and $\infty$, then we arrive at the definition of double ramification cycles instead. In \cite{fredDR} the first author will use a similar approach to describe the closure of double ramification loci inside $\calM_{g,n}$.

We now briefly define the Hurwitz spaces that we consider. Let $f:X\to\PP^1$ be a degree $d$ map, which we think of as a rational function, branched over $x_1,\dots,x_n\in\PP^1$, with local ramification indices $\left(e_{1}^{(i)},\dots,e^{(i)}_{k_i}\right)$ over $x_i$. We call the tuple $\underline{d}=\left(d;\left(e^{(1)}_{1},\dots,e^{(1)}_{i_1}\right),\dots, \left(e^{(n)}_{1},\dots,e^{(n)}_{i_n}\right)\right)$ the {\em branching profile} of $f$. For a fixed branching profile $\underline{d}$, we define the {\em Hurwitz space}
\begin{equation}\label{eq:Hurwdefine}
\Hur(\underline{d}):=\left\{
(X,\underline{z},f:X\to \PP^1)\,:\, \begin{array}{l} f \text{ has  branching profile } \underline{d},\\\mult_{z_k^{(i)}} f=e_k^{(i)},\,\\ f(z_k^{(i)})=f(z_{k'}^{(i)})\,\forall k,k',\,\\
\end{array}
\right \}/\sim
\end{equation}
where $\underline{z}=\left(z_1^{(1)},\dots,z_{i_1}^{(1)},\dots, z^{(n)}_{1},\dots,z^{(n)}_{i_n}\right)\subset X$ is a collection of distinct labeled points, and $\mult_{z} f$ denotes the ramification index of $f$ at $z$. Two such covers $(X,\underline{z},f),\, (X',\underline{z'},f')$ are considered equivalent if there exists an isomorphism $\phi:(X,\underline{z})\to (X',\underline{z'})$ of pointed Riemann surfaces, and an isomorphism $\psi:\PP^1\to\PP^1$ such that the diagram
\[
\begin{tikzcd}
(X,\underline{z})\ar[r,"\phi"] \ar[d,"f"]& (X',\underline{z}') \ar[d,"f'"]\\
\PP^1\ar[r,"\psi"] & \PP^1
\end{tikzcd}
\]
commutes. Note that for every $(X,\underline{z},f)$, after composition with an automorphism of~$\PP^1$, we can assume that $f^{-1}(\infty)=\left\{z_1^{(n)},\dots,z_{i_n}^{(n)}\right \}$. After this normalization we can still translate and rescale $f$. Since $df$ is unchanged when~$f$ is translated, a rational function up to automorphisms is the same as an exact differential up to rescaling.

Given a branching profile $\underline{d}$, we define a partition
\[
\mu=\left(\mu_1^{1},\dots,\mu^{(1)}_{i_1},\dots, \mu^{(n)}_1,\dots,\mu_{i_n}^{(n)}\right )
\]
of $2g-2$ by setting $\mu_k^{(i)}:=\ord_{z_k^{(i)}} df$, where we normalize as above, so that $f$ is assumed to have poles exactly at $z_1^{(n)},\dots,z_{i_n}^{(n)}$. Thus thinking of the triple $(X,\underline{z},df)$ instead of $(X,\underline{z},f)$ gives a map of the Hurwitz space to the projectivized stratum $\PP\Omega\calM_{g,n}(\mu)$, and we thus see that $\Hur(\underline{d})$ is isomorphic to the linear subvariety
\[
\PP\Omega\!\Hur(\underline{d}):=\left\{(X,\underline{z},\omega)\in \PP\calM_{g,n}(\mu)\,:\,\begin{array}{l} \int_{\gamma} \omega=0\,\text{   \ \ \  }\forall\gamma\in H_1(X;\ZZ),\,\\ \int_{p_k^{(i)}}^{p_{k'}^{(i)}} \omega =0\, \text{   }\forall k,k', i\neq n
\end{array} \right\}
\]
We can thus compactify $\Hur(\underline{d})$ by taking the closure of $\PP\Omega\!\Hur(\underline{d})$ inside $\PP\LMS$.
\begin{prop}\label{prop:HW}
The closure $\overline{\PP\Omega\!\Hur(\underline{d})}\subseteq\PP\LMS$ is a {\em smooth} compactification of $\PP\Omega\!\Hur(\underline{d})$.
\end{prop}
\begin{proof}
For any boundary point $p_0\in\partial{\PP\Omega\!\Hur(\underline{d})}$ we first claim that $\Gamma$ has no horizontal nodes. Indeed, at a horizontal node, the residue of a twisted differential is non-zero, but since the twisted differential is the rescaled limit of exact differentials $df$, all of whose absolute periods are zero on all flat surfaces in $\PP\Omega\!\Hur(\underline{d})$, this is impossible.

Thus by~\Cref{cor:smooth} every local irreducible component of $\overline{\PP\Omega\!\Hur(\underline{d})}$ at $p_0$ is smooth. Note that \Cref{cor:smooth} has been stated only for unprojectivized strata but the proof applies also for projectivized strata, as the extra local factor of $\CC^*$ does not make any difference.

It remains to show that $\overline{\PP\Omega\!\Hur(\underline{d})}$ is locally irreducible at $p_0$. Assume that $\overline{Z}_1$ and $\overline{Z}_2$ are two  local irreducible components  of $\overline{\PP\Omega\!\Hur(\underline{d})}$ near $p_0$, and choose smooth points $p_i=(X_i,\omega_i)\in Z_i$. Given a $\lG$-adapted basis at $p_1$ and a path $\gamma$ from $p_1$ to $p_2$, we can transport it to a $\lG$-adapted basis at $p_2$ using the Gauss-Manin connection along $\gamma$. The resulting homology basis at $p_2$ depends on $\gamma$, but the bases at $p_2$ obtained by translating along different paths will only differ by adding multiples of the vanishing cycles. Since every vanishing cycle is  contained in absolute homology, the resulting analytic equations for $\overline{Z_2}$ will be independent of the choice of~$\gamma$. Let $N$ be the defining linear equations for $\overline{Z}_1$ at $p_1$ in the chosen $\lG$-adapted basis, and let $\GM(N)$ be the result of transporting them along $\gamma$ using the Gauss-Manin connection. Then $\GM(N)$ are defining equations for $\overline{Z}_2$ at $p_2$, simply because they are again the equations of vanishing of all absolute periods, and the vanishing of relative periods (which is a condition that is independent of the choice of the path, as all absolute periods are zero).
Using \eqref{eq:convert} we can convert~$N$ and~$\GM(N)$ into the analytic equations of $\overline{Z}_1$ and $\overline{Z}_2$ near $p_0$ in plumbing coordinates. Since $N$ and $\GM(N)$ induce the same analytic equations near $p_0$,
 Thus at every boundary point there can be only one local irreducible component, and it finally follows that the closure is smooth.
\end{proof}

\section{Cylinder deformation theorem}\label{sec:cyldef}
In this section we use the restrictions on linear equations that we have obtained (in particular, \Cref{thm:decomp} and \Cref{thm:Mrltd}) to give a new proof of the Cylinder Deformation Theorem.  The key point is that we can decompose the defining equations in such a way that all the cylinders crossed by any particular equation are $M$-parallel.  It then follows that~$M$ admits some deformation changing just the cylinders in an $M$-parallel equivalence class, and in fact we show that stretching/shearing all these cylinders by the same matrix remains in~$M$.

Below, we will need to consider a \emph{cross-curve} $\delta$ of a cylinder $C$ on a flat surface $p_0$,  in the sense of Wright~\cite{wright}.  This is defined to be a curve represented by a saddle connection that lies in the cylinder, crosses the cylinder, and has one endpoint at a zero on the bottom boundary of the cylinder, and the other endpoint at a zero on the top boundary  (note that a cross-curve can cease to be a cross-curve under a small perturbation, for instance if the cylinder contains multiple zeroes on each of its boundary components).

We start with a Lemma that gives a connection between horizontal nodes and Euclidean cylinders for the flat metric of large modulus.
\begin{lm}  \label{lm:cyl-node}
For any $p_0\in\DG\subset\partial\LMS$, there exists a sufficiently small neighborhood $U\ni p_0$ and a sufficiently large $R>0$,  such that for any flat surface $p=(X,\omega)\in U\cap \omodulin(\mu)$, and for any flat Euclidean cylinder $C\subset X$ of modulus greater than $R$, the circumference curve~$\lambda$ of~$C$ is a horizontal vanishing cycle.
\end{lm}
\begin{proof}
We first show that the core curve of every essential annulus of sufficiently large modulus must be homotopic to a vanishing cycle. While this is an easy standard argument, we have not been able to pinpoint a precise reference in the literature. For this, we forget the flat structure, and work in a neighborhood $U$ of a nodal curve $X_0\in\partial\overline{\calM}_{g,n}$, where every smooth curve has a thick-thin decomposition, where we think of $U$ in terms of standard plumbing coordinates near the boundary of $\overline{\calM}_{g,n}$. Let $\lambda$ be the core curve of an essential annulus on~$X$ of sufficiently large modulus $R$. Then by the Schwarz lemma the homotopy class of $\lambda$ contains a short closed geodesic $\lambda'$ for the hyperbolic  metric on $X$ (where short means of length going to zero, as $R\to\infty$). We claim that $\lambda'$ cannot intersect the thick part of~$X$.

To this end, observe that the hyperbolic length of all closed geodesics on the thick part of all $X\in U$ is bounded below by a non-zero constant, and thus by increasing $R$ if necessary, we can ensure that it cannot happen that $\lambda'$ is contained in the thick part. If $\lambda'$ intersects both the thin and thick parts of~$X$, consider a ``shortened'' plumbing annulus, where collars of hyperbolic width 1 are fixed at both ends. Then by using this smaller plumbing neighborhood to start with, we can ensure that $\lambda'$ must intersect both the thin part in the shortened plumbing annulus, and the thick part for the original longer plumbing annulus. In particular, $\lambda'$ must cross from one boundary of the collar to the other, but then the hyperbolic length of $\lambda'$ must be at least 1, so $\lambda'$ cannot be short. Thus finally $\lambda'$ must be contained in the thin part, but then it must be contained in one plumbing annulus, and finally it must be homotopic to the corresponding vanishing cycle.

We now switch from this general discussion to the situation of essential annuli that are Euclidean cylinders for the flat metric. Choose a neighborhood~$U$ of $p_0\in \LMS$ sufficiently small so that every $p=(X,\omega)\in U$ is obtained by plumbing some $p'\in\DG$ under the plumbing construction of~\cite[Sec.~12]{BCGGMmsds}, and so that the above argument applies, for some chosen large~$R$. We claim that (possibly after further increasing~$R$ and shrinking $U$) the core curve of any Euclidean cylinder~$C$ is homotopic to a horizontal vanishing cycle.

Suppose for contradiction that the core curve of~$C$ is homotopic to some vertical vanishing cycle~$\van$, for $e\in\Ever$. Recall that the plumbing construction for flat differentials glues in a plumbing annulus ${\mathbb V}_e$ around a vertical node, such that $\omega$ on it has the standard form $\Omega_e$ given by~\cite[(12.7)]{BCGGMmsds}. In particular $\omega$ has no zeroes or poles on ${\mathbb V}_e$. The cross-curve $\delta$ of $C$ connects two zeroes of $\omega$ and thus must cross into the thick part of both $X_{(\lbot)}$ and $X_{(\ltop)}$.
In particular we can choose a geodesic $\lambda'$ for the flat metric on~$C$ that is in the isotopy class of $\van$ and passes through some fixed point $x$ in the thick part of $X_{\ltop}$.
Let $D$ be a small disk of fixed radius around $x$, contained in the thick part of~$X$. Then the length $|\int_\lambda\omega|=\int_{\lambda'}|\omega|$ of $\lambda'$ in the flat metric is bounded below by
\[
\int_{\lambda'\cap D}|\omega|= c\cdot |\scl[\ltop]|,
\] where $\scl[\ltop]$ is the scaling parameter for $\omega$ on $X_{(\ltop)}$ and $c$ is a constant independent of $\omega$, which depends on the size of $D$ and the choice of the thick-thin decompositions. Note that $c$ depends on which zeros are connected by the cross-curve $\delta$, but since there are only finitely many zeros we can choose $c$ to be the minimum.

On the other hand, $\lambda_e$ is homotopic to a path contained in the thick part of $X_{(\lbot)}$ and thus the length of $\lambda_e$ can be bounded above by $c'\cdot |\scl[\lbot]|$ for some constant $c'$ independent of $\omega$. This contradicts the fact that $\left| \scl[\ltop] \right| \gg \left| \scl[\lbot] \right|$ on $U$ (after possibly further shrinking $U$).
\end{proof}

We are now ready to prove our generalization of the cylinder deformation theorem.
\begin{proof}[Proof of~\Cref{thm:cyldeformation}]
We begin by using the $\GL^+(2,\RR)$-action to get to a boundary point of~$M$ where we can apply our results restricting the defining equations of~$M$. Recall that elements $a_t,u_s\in\GL^+(2,\RR)$ can be applied to any flat surface, and that $a_t$ and $u_s$ preserve~$M$, since~$M$ is given by linear equations with real coefficients. Recall that transformations $a_t^\calC$ and $u_s^\calC$ only act on cylinders in the class $\calC$, leaving the rest of the flat surface unchanged, and our goal is to show that they also preserve~$M$.

For any given $(X,\omega)$ the forward orbit $\{a_t (X,\omega)\}_{t\ge 0}$ is contained in~$M$. Since all cylinders in $\calC$ are stretched unboundedly by $a_t$ as $t\to+\infty$, the underlying Riemann surfaces in this orbit degenerate as $t\to +\infty$.  Thus the image of this forward orbit in the projectivization $\PP\omodulin(\mu)$ cannot be compact, and there must exist a boundary point $\PP p_0=\PP(X_0,\lG,\eta_0)\in \partial\PP\LMS$
and a sequence $\{ t'_n\}$ of positive numbers such that $t'_n\to+\infty$ and $ \PP a_{t'_n} (X,\omega)\to \PP p_0$. Here by $\PP p_0$ we mean the image in $\PP\LMS$ of a point $p_0\in\LMS$, as throughout the paper, under the quotient map $\LMS\to\PP\LMS$.  This implies that there exist complex numbers $r_n$ such that $r_n a_{t'_n}(X,\omega) \to p_0$.  By taking a subsequence, we can assume that the angles of the $r_n$ converge to some $\alpha\in S^1$.  We now rotate each $r_n$ so that it is positive and real, and we replace $p_0$ with $-\alpha p_0$.   In the end we get a sequence $r_n$ of positive reals so that $r_n a_{t'_n}(X,\omega) \to p_0$.

By throwing away some of the beginning terms of the sequence, we can assume that all $r_n a_{t'_n}(X,\omega)$ lie in $U$ (recall that $U$ is a small neighborhood of the boundary point $p_0$).  Denote $(Y,\omega_Y):=a_{t'_1} (X,\omega)$.  It suffices to prove the statement for $(Y,\omega_Y)$ instead of $(X,\omega)$.

We now subdivide $U$ into a finite number of simply connected sets (as in ~\cite[Section 3.1]{ben} or \cite [Section 8]{chwr}).  By passing to a subsequence, we can assume that $r_n a_{t'_n}X$ all lie in one of these sets $W$.  Let $\{t_n:=t'_n-t'_1\}$ be the sequence such that $a_{t_n}Y=a_{t'_n}X$.
Below we will not need to think of the $1$-form separately from the Riemann surface, so we will drop $\omega_Y$ from the notation; we will denote $\beta(Y):=\int_\beta \omega_Y$ the period over a relative homology class~$\beta$.

By \Cref{lm:cyl-node}, the circumference curve of each horizontal cylinder $C_i$ on $a_{t_n}Y$ must be the vanishing cycle for some horizontal node of $p_0$ (note that we are abusing notation by thinking of $C_i$, initially defined to be a cylinder on $X$, as a cylinder on $a_{t_n}Y$; this creates no issues since all of these surfaces are in the $a_t$ orbit of $X$).  By passing to a further subsequence of $t_n$, we can assume that the horizontal node $e_i^{(n)}\in\Ehor$ whose vanishing cycle is the circumference curve of $C_i$ on $a_{t_n}Y$ does not in fact depend on $n$.

By~\Cref{thm:decomp}, any defining equation~$F$ of~$M$ can be decomposed as $F=H_1+\cdots+H_k +G$, where each $H_j$ crosses a primitive collection of horizontal nodes, all at level $\topl(H_j)$, and $G$ does not cross any horizontal nodes.  To show that $a_t^\calC u_s^\calC Y\in M$, it is thus enough to show that any such defining equations $H_1,\ldots, H_k, G$ vanish also at the point $a_t^\calC u_s^\calC Y$.

We will express this deformation in terms of periods of cross-curves (see the discussion at the beginning of this section for the definition).  Let $\delta_i$ be a cross-curve of the cylinder $C_i$ on the surface $Y$.  Since $C_i$ is horizontal, its height $h_i(Y)$ on the surface $Y$ equals $\Im\, \delta_i(Y)$.   Note that $\delta_i$ can be thought of as a relative homology class on all surfaces in the simply connected set $W$.  For sufficiently small $t,s$, the deformation $a_t^\calC u_s^\calC$ changes periods by
\begin{equation}\label{eq:def}
  \delta_i(Y) \mapsto a_t u_s  (\delta_i(Y)),
\end{equation}
for any~$i$, while preserving the period over any curve that does not cross any cylinder $C_i$. Since the class $G$ does not cross any horizontal nodes, and in particular does not cross the circumference curve of any $C_i$, it follows that $G(a_t^\calC u_s^\calC Y)=G(Y)=0$ for any sufficiently small $t,s$.

For a defining equation $H_j$, first note that if it does not cross any of the cylinders in $\calC$, then it is similarly preserved under the deformation~$a_t^\calC u_s^\calC$. Suppose now that $H_j$ crosses some $C_i\in\calC$. Since the collection of horizontal nodes that $H_j$ crosses is primitive, all of these nodes are \Mrel.  Hence by~\Cref{thm:Mrltd}, the periods over all the vanishing cycles crossed by $H_j$ are proportional on~$M$, and hence all the cylinders crossed by $H_j$ are $M$-parallel.  Since $\calC$ is a full equivalence class of $M$-parallel cylinders, all of the cylinders crossed by $H_j$ must lie in $\calC$.

We can thus write
\begin{align}\label{eq:horiz-eqns}
   H_j = \beta_j+ \sum_{i=1}^d  c_{i,j} \delta_i,
\end{align}
where $\beta_j$ is a relative homology class that does not cross any horizontal nodes, $\topl(\beta_j) \le \topl(H_j)$, and $c_{i,j}$ are some real numbers. Furthermore, $\topl(\delta_i) = \topl(H_j)$ for all~$i$, since $H_j$ has the same top level as the cylinders in $\calC$, and $\delta_i$ is a cross-curve of such a cylinder.

Without the $\beta_j$ term, $a_t^\calC u_s^\calC$ would act on $H_j$ in exactly the same way that $a_tu_s$ does, and $H_j(a_t^\calC u_s^\calC Y) =0$ would follow from the fact that~$M$ is defined by linear equations with real coefficients, which are preserved by the $\GL^+(2,\RR)$ action.   The presence of the $\beta_j$ term makes the proof more complicated.  We will use our sequence $a_{t_n}Y$ to prove the following

{\bf Claim:} The imaginary part $\Im\, \beta_j(Y) $ is zero.

Assuming the claim, the fact that $H_j(a_t^\calC u_s^\calC Y)=0$ follows easily. Indeed, we first compute the difference
\begin{align*}
  H_j&(a_t^\calC u_s^\calC Y) - a_tu_s H_j(Y)\\ &= \left(\beta_j(a_t^\calC u_s^\calC Y) + \sum_i c_{i,j} \delta_i(a_t^\calC u_s^\calC Y) \right) - \left(a_tu_s\beta_j(Y) + a_tu_s\sum_i c_{i,j}\delta_i(Y)\right) \\
  &=\left(\beta_j(Y) + a_tu_s\sum_i c_{i,j} \delta_i(Y) \right) - \left(a_tu_s\beta_j(Y) + a_tu_s\sum_i c_{i,j}\delta_i(Y)\right) \\
  &=\beta_j(Y) - a_tu_s(\beta_j(Y)) = 0,\\
\end{align*}
where in the last equality we used the Claim: since $\Im\,\beta_j(Y)=0$,  one computes $a_t u_s(\beta_j(Y))=\beta_j(Y)$.  Now since $H_j(Y)=0$, we have $a_tu_sH_j(Y)=0$, hence the above implies $H_j(a_t^\calC u_s^\calC Y)=0$, as desired.

To complete the proof of the Theorem, it thus remains to prove the Claim, for which we will use the convergent sequence $r_na_{t_n}Y \to p_0$ constructed in the beginning of the proof.  Since $r_na_{t_n} Y\in M$, we know that $H_j(r_na_{t_n} Y)=0$. Taking the imaginary part and using the expression \eqref{eq:horiz-eqns} gives
\begin{equation}  \label{eq:imag}
    \Im (\beta_j(r_n a_{t_n} Y)) + \sum_i c_{i,j} \Im(\delta_i(r_na_{t_n} Y))  =0
\end{equation}
(note that this again uses the fact that we are working with a linear variety defined by equations with real coefficients, so that $c_{i,j}$ are real).

The curve~$\delta_i$ on $r_na_{t_n}Y$ is a curve that crosses the cylinder~$C_i$. While~$\delta_i$ on~$r_na_{t_n}Y$ is not necessarily a cross-curve in the sense above, we claim that it has the same top level as the vanishing cycle of $C_i$. Indeed, to see this one argues as in the proof of \Cref{lm:cyl-node}: if $\delta_i$ had higher top level, then one could choose a closed geodesic representing the circumference curve of~$C_i$ that would cross the thick part of the surface at this higher level, which would then have length much larger than the magnitude of the period over the vanishing cycle; on the other hand, since $\delta_i$ crosses  $C_i$, its top level is at least the level of~$C_i$.   Since $C_i$ is a horizontal cylinder, its height $h_i(r_na_{t_n}Y)=r_ne^{t_n}h_i(Y) =r_ne^{t_n} \Im\, \delta_i(Y)$ on the surface~$r_na_{t_n}Y$ is approximated by $\Im\, \delta_i(r_na_{t_n}Y)$; in fact
\begin{align*}
  \Im\, \delta_i(r_na_{t_n}Y) - r_ne^{t_n} \Im\, \delta_i(Y) = o(r_n e^{t_n}),
\end{align*}
  as $n\to\infty$.
Substituting this into \eqref{eq:imag} gives
\begin{equation*}
  \Im\, \beta_j(r_na_{t_n} Y) + \sum_i c_{i,j}\left( r_ne^{t_n} \Im\, \delta_i(Y)+ o(r_n e^{t_n})\right)  = 0,
\end{equation*}
and dividing through by $r_ne^{t_n}$ gives
\begin{equation*}
  \frac{\Im(\beta_j(r_na_{t_n} Y))}{r_ne^{t_n}} + \sum_i c_{i,j}\Im\, \delta_i(Y) = 0.
\end{equation*}
Recall that $\beta_j$ is a curve with top level $\topl(\beta_j) \le \topl(H_j)$, while $\topl(H_j)$ is the level of the circumference curve of each of the cylinders $C_i$.  Hence, on surfaces in $W$, the magnitude of the period of $\beta_j$ is less than a constant multiple of the circumference of $C_i$.   The height of each $C_i$ on $r_na_{t_n}Y$ is within a constant factor of $r_ne^{t_n}$.  Each cylinder $C_i$ is degenerating as $n\to\infty$, so its modulus is going to infinity.  It follows that the left-hand term in the above goes to $0$ as $n\to \infty$.
Taking the limit, we get
\begin{equation*}
  \sum_i c_{i,j} \Im\, \delta_i(Y) = 0.
\end{equation*}
Combining this with the fact that at $Y$ the imaginary part of \eqref{eq:horiz-eqns} is $0$, we get $\Im\,\beta_j(Y))=0$, as claimed.
\end{proof}

\section{The linear equations of affine invariant submanifolds}\label{sec:AIM}
In this section we specialize our study of linear subvarieties to the case of affine invariant submanifolds. In our language, this is simply to say that we are talking about linear subvarieties of holomorphic strata (so all $m_i>0$) such that furthermore all the defining linear equations have real coefficients.
Avila, Eskin and M\"oller \cite{aem} show that for any affine invariant manifold~$M$ in a holomorphic stratum the image of the tangent space $T_{(X,\omega)}M\subset H^1(X,\zeroes;\CC)$ in $H^1(X;\CC)$ is symplectic under the natural symplectic pairing. We first carefully set up notation for all this.
\subsection{General setup}
\label{sec:holom-strata}
Denote
\begin{align*}
  \iota: H_1(X;\CC) \mathop{\inj} H_1(X,\zeroes;\CC),\\
  u: H_1(X\setminus \zeroes; \CC) \mathop{\surj} H_1(X;\CC)
\end{align*}
the natural maps, and by abuse of notation denote by $\langle,\rangle$ both natural intersection pairings
\begin{align*}
  H_1(X;\CC) \times H_1(X;\CC) \to \CC, \\
  H_1(X,\zeroes;\CC) \times H_1(X\setminus \zeroes; \CC) \to \CC,
\end{align*}
which satisfy the adjunction property
\begin{align*}
  \langle x, u(v) \rangle = \langle \iota(x), v\rangle
\end{align*}
for any $x\in H_1(X;\CC)$, $v\in H_1(X\setminus \zeroes;\CC)$.
Given a subspace $V\subset H_1(X;\CC)$ (or of $H_1(X,\zeroes;\CC)$, respectively $H_1(X\setminus \zeroes;\CC)$), we denote by $V^{\perp}$ the perp space with respect to $\langle,\rangle$ in $H_1(X;\CC)$ (respectively, of $H_1(X\setminus \zeroes;\CC)$ or $H_1(X,\zeroes;\CC)$). For a subspace $V$ of homology, we denote by $\Ann V$ its annihilator in cohomology.

The following result controls the space of deformations in~$M$ supported on an equivalence class of $M$-parallel cylinders, modulo purely relative deformations.  In the below we will study small deformations, which are elements of the tangent space $TM\subset H^1(X,\zeroes;\CC)$ at the given point $(X,\omega)$.
\begin{lm}
  \label{lm:proj-def}
  Let $M$ be an affine invariant manifold, and let $\calC$ be an equivalence class of $M$-parallel cylinders on some $(X,\omega)\in M$.  Let $V\subseteq H_1(X\setminus \zeroes;\CC)$ be the span of the circumference curves of the cylinders in $\calC$.  Then
  $$ \dim \iota^*\left(T M \cap \Ann V^{\perp}\right) \le 1.$$
\end{lm}

Note that $V^{\perp}$ consists of all homology classes that don't intersect one of the cylinder circumference curves in $\calC$, so $\Ann V^{\perp}$ is the space of local deformations in the stratum supported on the union of these cylinders.  Hence one can think of $\iota^*\left(T M \cap \Ann V^{\perp}\right)$ as local deformations in $M$ supported on the union of cylinders in $\calC$, modulo purely relative deformations.

The proof of the lemma is a simple application of the result of Avila-Eskin-M\"oller~\cite{aem} that $\iota^*(TM)$ is symplectic, together with some formal linear algebra.
\begin{proof}
We first note that in our linear algebra setup, for any subspace $W\subseteq H_1(X\setminus \zeroes, \CC)$, we have
  \begin{align*}
    \iota\left( u(W)^\perp\right) \subseteq W^\perp,
  \end{align*}
while for any subspace $Z\subseteq H_1(X;\CC)$, we have
  \begin{align*}
    \iota^* \left( \Ann (\iota(Z))\right) \subseteq \Ann(Z).
  \end{align*}

  Using these two facts, we get
  \begin{align}
    \iota^*\left(T M \cap \Ann V^{\perp}\right) &\subseteq \iota^*(TM) \cap \iota^*\left(\Ann V^{\perp}\right) \\
        & \subseteq \iota^*(TM) \cap \iota^*\left(\Ann \left(\iota (u(V)^\perp)\right)\right)\\
    & \subseteq \iota^*(TM) \cap \Ann \left(u(V)^{\perp}\right).     \label{eq:inclusion}
  \end{align}

The above are subspaces of absolute cohomology, but the symplectic form is easier to understand in absolute homology, so we take the annihilator of the last term above:
 \begin{align*}
   \Ann\left(\iota^*(TM) \cap \Ann (u(V)^{\perp})\right) &= \Ann (\iota^*(TM)) + \Ann \left(\Ann (u(V)^{\perp})\right)\\
   & = \Ann (\iota^*(TM)) + u(V)^{\perp}.
 \end{align*}
 By this equality and \eqref{eq:inclusion}, to prove the Lemma it suffices to show that
 \begin{align}
   \label{eq:dual-ver}
   \dim \left(\Ann (\iota^*(TM)) + u(V)^{\perp}\right) \ge n-1,
 \end{align}
 where $n:=\dim H_1(X;\CC)$.

Now recall that $V$ is spanned by circumference curves of $M$-parallel cylinders, which is to say they remain parallel under small deformations in $TM$. Thus annihilating them imposes only one condition on $TM$, which is to say we have
\begin{align}
   \label{eq:cap}
   \dim \left( \Ann (\iota^* TM ) \cap u(V) \right) = \dim u(v) -1,
 \end{align}
By~\cite{aem}, $\iota^*(TM)\subset H^1(X;\CC)$ is symplectic subspace, and hence so is $\Ann(\iota^*(TM))$.  It follows that
\begin{align*}
  \dim ( u(V) \cap \Ann (\iota^*(TM))) +   \dim (u(V)^\perp \cap \Ann (\iota^*(TM))) = \dim \Ann(\iota^*(TM)).
\end{align*}

Combining this with \eqref{eq:cap} gives
\begin{align}
  \label{eq:7}
  \dim \left(u(V)^\perp \cap \Ann (\iota^*(TM))\right) =  \dim \Ann(\iota^*(TM)) - \dim u(v) + 1
\end{align}

Thus
\begin{align*}
  \dim &\left(u(V)^\perp + \Ann(\iota^*(TM)) \right) \\
       &= \dim u(V)^\perp + \dim \Ann(\iota^*(TM)) - \dim \left(u(V)^\perp \cap \Ann (\iota^*(TM))\right)\\
       &= \dim u(V)^\perp + \dim \Ann(\iota^*(TM)) - \left(  \dim \Ann(\iota^*(TM)) - \dim u(v) + 1 \right)\\
       &= \dim u(V)^\perp + \dim u(v) -1 = n-1,
\end{align*}
which establishes \eqref{eq:dual-ver}, so we are done.
\end{proof}

We can now easily reprove a partial converse to the cylinder deformation theorem, originally proved by Mirzakhani-Wright \cite[Theorem 1.5]{miwr}.
\begin{thm}
  \label{thm:converse-cyl-gen}
Let~$M$ be an affine invariant manifold in any holomorphic stratum.  Let $(X,\omega)\in M$, and let $\calC$ be a full equivalence class of horizontal $M$-parallel cylinders.  Then, up to purely relative deformations, the only small deformations of $(X,\omega)$ that stay in~$M$ and are supported on the union of the cylinders in $\calC$ are given by $a_t^\calC u_s^\calC(X,\omega)$, for small $t,s\in \RR$.
\end{thm}

\begin{proof}
  By \Cref{lm:proj-def}, the space of such deformations is at most one-dimensional.  By the cylinder deformation theorem, \Cref{thm:cyldeformation}, the deformation given by applying $a_t^\calC u_s^\calC$ lies in~$M$.  Hence this one-complex dimensional family comprises all deformations of the specified type.
\end{proof}

For further use, we record the following easy statement
\begin{lm}   \label{lm:hom-indep}
The set of all horizontal vanishing cycles is linearly independent in punctured homology $H_1(X\setminus\zeroes;\CC)$.
\end{lm}
\begin{proof}For each horizontal vanishing cycle $\van$,  we can choose a homology class $\delta\in H_1(X,\zeroes;\CC)$ that intersects $\van$ and no other horizontal vanishing cycle.
For instance $\delta$ can be constructed by locating two marked zeros at levels below $\ell(e)$ that can be connected by a path that crosses no horizontal nodes except $e$, and is contained in $X_{(\le\ell(e))}$. The existence of such marked zeros is guaranteed by \cite[Lemma 3.9]{BCGGMivc}.
\end{proof}
We note that it is not true that the set of all vanishing cycles altogether is linearly independent in punctured homology. Indeed, if some irreducible component of the multi-scale differential does not contain any marked zero, then the sum of the vanishing cycles that is its boundary is homologous to zero.

\subsection{Minimal stratum}
\label{sec:min-strata}
We now specialize to the case of affine invariant manifolds in a \emph{minimal} holomorphic stratum $\Omega\calM_{g,1}(2g-2)$, i.e.~to the case when the differential has only zero, of maximal multiplicity. The special feature of the minimal stratum is that both maps $\iota$ and $u$ above are isomorphisms; in particular all horizontal vanishing cycles are linearly independent in the absolute homology $H_1(X;\CC)$. As always, we study the situation near some $p_0\in \pM\subset\Xi\overline{\calM}_{g,1}(2g-2)$, and the first result we obtain is the following.

\begin{prop}  \label{prop:pairwise-cross}
Let $e_1\ne e_2\in\Ehor$ be \Mrel horizontal nodes.  Then there is a defining equation $F$ of~$M$ that crosses $e_1,e_2$, and no other horizontal nodes, i.e $\Ehor[F]=\lbrace e_1,e_2\rbrace$.
\end{prop}

\begin{proof}
Let $\lambda_1,\lambda_2$ be the vanishing cycles for $e_1$, $e_2$, and let $\Lambda$ be the \Mequiv class containing them. Let $W$ be the span of the elements of $\Lambda$.    By ~\Cref{lm:hom-indep}, $\dim W = |\Lambda|$.  By \Cref{thm:Mrltd}, the vanishing cycles in $\Lambda$ all have proportional periods on $M$, and so the corresponding cylinders are $M$-parallel.  Now let $V$ be the span of the vanishing cycles of all cylinders that are $M$-parallel to these.  Since $W\subset V$, by \Cref{lm:proj-def} (and using that the map $\iota$ is an isomorphism, since we are working in the minimal stratum), we get
  \begin{align*}
    \dim (TM \cap \Ann W^{\perp})  \le \dim (TM \cap \Ann V^{\perp}) \le 1.
  \end{align*}
Hence
\begin{align*}
  \dim \Ann W^{\perp} - \dim (TM \cap \Ann W^{\perp}) \ge \dim \Ann W^{\perp}  -1 = |\Lambda| -1.
\end{align*}

The left-hand side above is equal to the number of equations in the rref basis that cross some vanishing cycle in $\Lambda$.  Since no equation can cross exactly one element of $\Lambda$, we get the desired conclusion.
\end{proof}

\begin{prop}  \label{prop:pairwise-circum}
Suppose $F=a_1\lambda_1 + \cdots + a_n \lambda_k$ is a defining equation of~$M$ at~$p$, where the $\lambda_i$ are some distinct horizontal vanishing cycles.  Then $F$ is a sum of defining equations of~$M$ that have the form of pairwise proportionalities $b_j \lambda_j =c_l \lambda_l$, for $1\le j,l\le k$, and some $b_j,c_l\in\RR$.
\end{prop}
\begin{proof}
Let $\lambda_1,\ldots, \lambda_k, \lambda_{k+1},\ldots,\lambda_n$  be all the horizontal vanishing cycles. By \Cref{lm:pairs-gen}, we can write
\begin{align*}
    F=F_1 + \cdots + F_\ell,
\end{align*}
where each $F_i$ is a defining equation of $M$ of the form $F=b\lambda_j + c\lambda_l$.  Order the $F_i$ in such a way that $F_1,\ldots, F_{\ell'}$ are of the form $b\lambda _j + c\lambda_l$ with $1\le j,l\le k$, and the remaining equations $F_{\ell'+1},\ldots,F_\ell$ are not of this form.  If $\ell=\ell'$, then we are done, so suppose $\ell > \ell'$ and recall that $\lambda_1,\ldots,\lambda_n$ are linearly independent by~\Cref{lm:hom-indep}.  If all the $F_{\ell'+1},\ldots,F_\ell$ are of the form $ b \lambda_j + c\lambda_j$ with both $j,l>k$, then since $F$ itself does not have any $\lambda_i$ terms with $i>k$, we must also have $F=F_1+\cdots+F_{\ell'}$, and we are done.  Otherwise, we can assume that some equation, without loss of generality $F_{\ell'+1}$, is of the form $b\lambda_j + c\lambda_l$, with $j\le k<l$.  In this case, since in the sum $F=F_1+\cdots + F_{\ell}$, the $\lambda_l$ terms must cancel out,  some other equation, without loss of generality $F_{\ell'+2}$, must have the form $b'\lambda_{j'} + c' \lambda_l$.  We can then write the following new decomposition of $F$ into defining equations:
  \begin{align*}
    F &= F_1 + \cdots + F_{\ell'} + \left(F_{\ell'+1} +\frac{c'}{c} F_{\ell'+1} \right) + \left(F_{\ell'+2} -\frac{c'}{c} F_{\ell'+1} \right) + F_{\ell'+3} + \cdots + F_\ell \\
      &=F_1 + \cdots + F_{\ell'} +\left(1+\frac{c'}{c}\right) F_{\ell'+1} +\left(b'\lambda_{j'} - \frac{c'b}{c}\lambda_j\right) + F_{\ell'+3} + \cdots + F_\ell.
  \end{align*}
Note that the $(\ell'+2)$-th term now involves $\lambda_j$ instead of $\lambda_l$.  Thus the total number of appearances of terms $\lambda_l$ with $l>k$ has decreased by $1$.  Continuing in this fashion, we arrive at a decomposition with no such terms, and we are done.
\end{proof}

The proof above used the following statement
\begin{lm}  \label{lm:pairs-gen}
Let $\lambda_1,\ldots, \lambda_n$  be all the horizontal vanishing cycles.  Then any defining equation of~$M$ of the form $F=a_1\lambda_1 + \cdots + a_n \lambda_n=0$ is a sum of defining equations of~$M$ of the form $b \lambda_j + c \lambda_l$.
\end{lm}

\begin{proof}
Consider a rref basis of defining equations of~$M$.  Suppose that among these equations, $F_1,..., F_k$ are the ones that cross some vanishing cycle among $\lambda_1,\ldots,\lambda_n$. We first claim that each such $F_i$ must cross at least two of the these horizontal vanishing cycles.  In fact, by \Cref{prop:monodromy}, there is a linear relation among the vanishing cycles crossed by $F_i$.   This could include vertical vanishing cycles, but by \Cref{prop:eqsquantitative}(1), these all lie at lower level than the horizontal nodes crossed.  Considering the limit of this relation as we approach $p_0$, and noting that the period of a horizontal vanishing cycle must be non-zero near $p_0$, we see that the relation must involve at least two horizontal vanishing cycles, as claimed.

Now for each equation $F_i$, consider the pivot horizontal vanishing cycle $\lambda_j$ for that equation, and choose some other horizontal vanishing cycle $\lambda_l$ crossed by $F_i$ (whose existence was just established).  By \Cref{thm:Mrltd}, there is some equation $\alpha_i=b\lambda_j + c\lambda_l$ that holds on $M$.   The $\alpha_1,\ldots,\alpha_k$ must be linearly independent, since each involves a pivot node that doesn't appear in the others.  Thus we get $k$ linearly independent relations, each establishing that periods over a pair of horizontal vanishing cycles are proportional.

Now let $\Lambda = \Span (\lambda_1,\ldots,\lambda_n)$.  By \cite{aem}, the tangent space $TM\subseteq H^1(X;\CC)$ is symplectic (here we again use that we are in the minimal stratum, so absolute and relative homology are the same).   Hence, $\Ann TM$ is also symplectic, which implies that
\begin{align*}
    \dim \Lambda \cap \Ann TM = \dim \Ann TM - \dim \Lambda^{\perp} \cap \Ann TM.
\end{align*}
We see that the right hand side equals $k$, since by assumption there are exactly $k$ rref equations cutting out~$M$ that have non-zero intersection with an element of $\Lambda$.  All $\alpha_1,\ldots,\alpha_k$ lie in $\Lambda \cap \Ann TM$, and since they are linearly independent we have
\begin{align*}
    \Lambda \cap \Ann TM = \Span (\alpha_1,\ldots,\alpha_k),
  \end{align*}
  and we are done, since $F$ also belongs in the left-hand space.
\end{proof}

\begin{proof}[Proof of \Cref{thm:AIM}]
  Part (1) will follow easily from \Cref{prop:pairwise-cross}.  We argue by induction on $n$, the number of horizontal nodes crossed by the defining equation $F$.  If $n\le 2$,  there is nothing to prove.  So suppose $n\ge 3$.  First, consider the case there exists a defining equation $F'$ with $\Ehor[F']\subsetneq \Ehor[F]$ and $\Ehor[F']\ne \emptyset$.  Subtracting from $F$ a suitable multiple $cF'$ gives an equation $F''$ that crosses a strictly smaller set of horizontal nodes than $F$.  Applying the inductive hypothesis to $F',F''$ gives that each can be written as a sum of defining equations that cross at most two horizontal nodes, and hence so can $F=cF'+F''$.  On the other hand, in the case that no such $F'$ exists, then any pair of horizontal nodes crossed by $F$ are \Mrel, by definition.  Pick two of these nodes and call them $e_1,e_2$.  By \Cref{prop:pairwise-cross}, there is a defining equation $F_0$ with $\Ehor[F_0] = \{e_1,e_2\}$.  Subtracting from $F$ a suitable multiple $cF_0$ gives an equation $F''$ that crosses a strictly smaller set of horizontal nodes than $F$.  By induction, this $F''$ can be written as a sum of defining equations that cross at most two horizontal nodes, and hence so can $F=cF_0+F''$.   This proves part (1).

  Part (2) is just a restatement of \Cref{prop:pairwise-circum}.
\end{proof}

\subsection{Counterexamples in general holomorphic strata}
\label{sec:counterex}

In this section, we show that analogs of certain results proved in \Cref{sec:min-strata} fail in general holomorphic strata with multiple zeros.

\begin{exa}
  \label{exa:counter-pairwise-circum}
  We give an example that shows that \Cref{prop:pairwise-circum} fails in general holomorphic strata.  Define an affine invariant manifold $M \subset \Omega\calM_{5,8}(1^8)$ to consist of all surfaces obtained as degree two branched covers of a surface in $\Omega\calM_{2,2}(1,1)$, branched over $4$ points none of which is a zero of the differential, with topological data as in \Cref{fig:RelThree}.   Then $M$ admits a degeneration $p_0$ where six classes in punctured homology $\lambda_1,\lambda_{1'}, \lambda_2,\lambda_{2'}, \lambda_3,\lambda_{3'}$ are vanishing cycles of horizontal nodes, as shown in \Cref{fig:RelThree}.  Near this boundary point, $M$ is cut out by the $8$ period equations shown next to the figure.

  Now note that taking the sum of the first three equations, and using that the relation $0 = \lambda_1+\lambda_2+\lambda_3+\lambda_{1'}+\lambda_{2'}+\lambda_{3'}$ holds in absolute homology, we see that
  \begin{align*}
    \int_{2\lambda_1+2\lambda_2 + 2\lambda_3} \omega = 0
  \end{align*}
is an equation satisfied on $M$.  However, no pair from $\lambda_1,\lambda_2,\lambda_3$ need to have proportional periods on $M$, so the above equation cannot be written as a sum of pairwise equations among these three vanishing cycles.  Hence \Cref{prop:pairwise-circum} fails for this $M$.

\begin{figure}
  \begin{minipage}[t]{0.55\linewidth}
    \vspace{0pt}
    \centering
    \includegraphics[width=\linewidth]{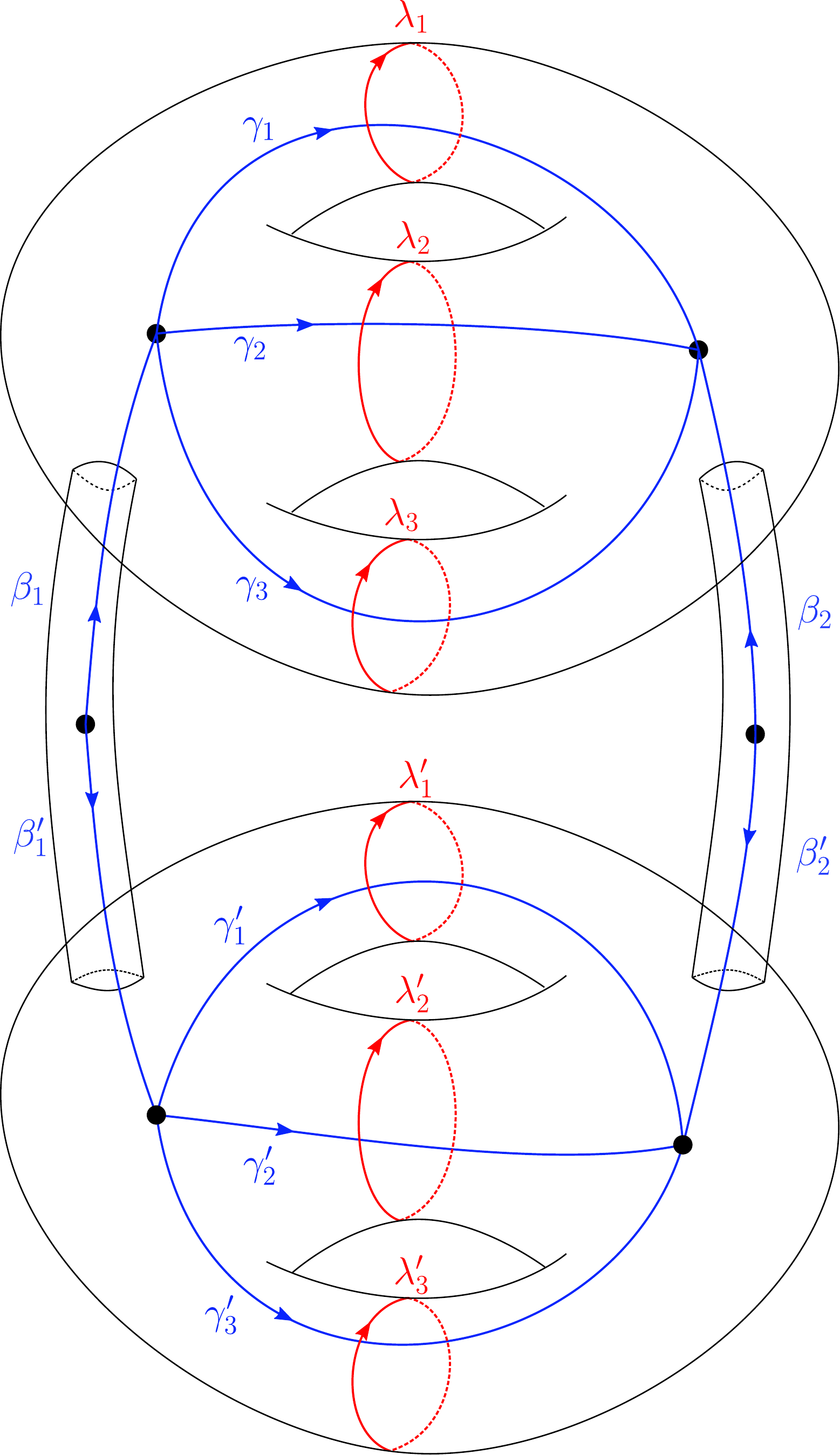}
  \end{minipage}
  \begin{minipage}[t]{0.50\linewidth}
    \vspace{0pt}
    \centering
    \renewcommand{\arraystretch}{1.7}
    \begin{tabular}{c}
      \hskip -1cm Period equations cutting out $M$ \\\hline
      $\int_{\lambda_1} \omega - \int_{\lambda_1'} \omega = 0$ \\
      $\int_{\lambda_2} \omega - \int_{\lambda_2'} \omega = 0$ \\
      $\int_{\lambda_3} \omega - \int_{\lambda_3'} \omega = 0$ \\
      $\int_{\gamma_1} \omega - \int_{\gamma_1'} \omega = 0$ \\
      $\int_{\gamma_2} \omega - \int_{\gamma_2'} \omega = 0$ \\
      $\int_{\gamma_3} \omega - \int_{\gamma_3'} \omega = 0$ \\
      $\int_{\beta_1} \omega - \int_{\beta_1'} \omega = 0$ \\
      $\int_{\beta_2} \omega - \int_{\beta_2'} \omega = 0$
    \end{tabular}
  \end{minipage}
  \caption{Description of an affine invariant manifold $M$ in the stratum $\Omega\calM_{5,8}(1^8)$}
  \label{fig:RelThree}
\end{figure}

\end{exa}

\begin{exa}
    \label{exa:counter-pairwise-cross}
We now give a \emph{local} counterexample $M$ to \Cref{prop:pairwise-cross} in a non-minimal stratum.  The meaning of local here is that there exists $p_0\in\partial\LMS$, and a neighborhood $U\ni p_0$ with a subvariety $M\subset U\cap\omodulin(\mu)$ such that:
  \begin{enumerate}
  \item at each its point $M$ is locally cut out by linear equations in period coordinates, \label{item:lin}
  \item the closure $\oM$ in $U$ is an analytic subvariety of $U$, containing $p_0$ \label{item:an}
  \item $\iota^*(TM)$ is a symplectic subspace.  \label{item:symp}
  \end{enumerate}
In other words, our $M$ will be consistent with all the local analyticity and symplecticity properties used in the proofs above, but we do not claim that~$M$ is actually the intersection of some (global) affine invariant manifold with~$U$.

We define $M\subset \Omega\calM_{3,3}(1,1,2)$ according to \Cref{fig:NoPairwise}.  Here the boundary point has three horizontal nodes, with vanishing cycles $\lambda_1, \lambda_2, \lambda_3$, and no other nodes.  (This does not specify a unique boundary point, since there is additional data about the geometry of the stable Riemann surface with differential.  Any choice of this data will do.)

\begin{figure}[h]
  \begin{minipage}[t]{0.8\linewidth}
    \vspace{0pt}
    \centering
    \includegraphics[width=\linewidth]{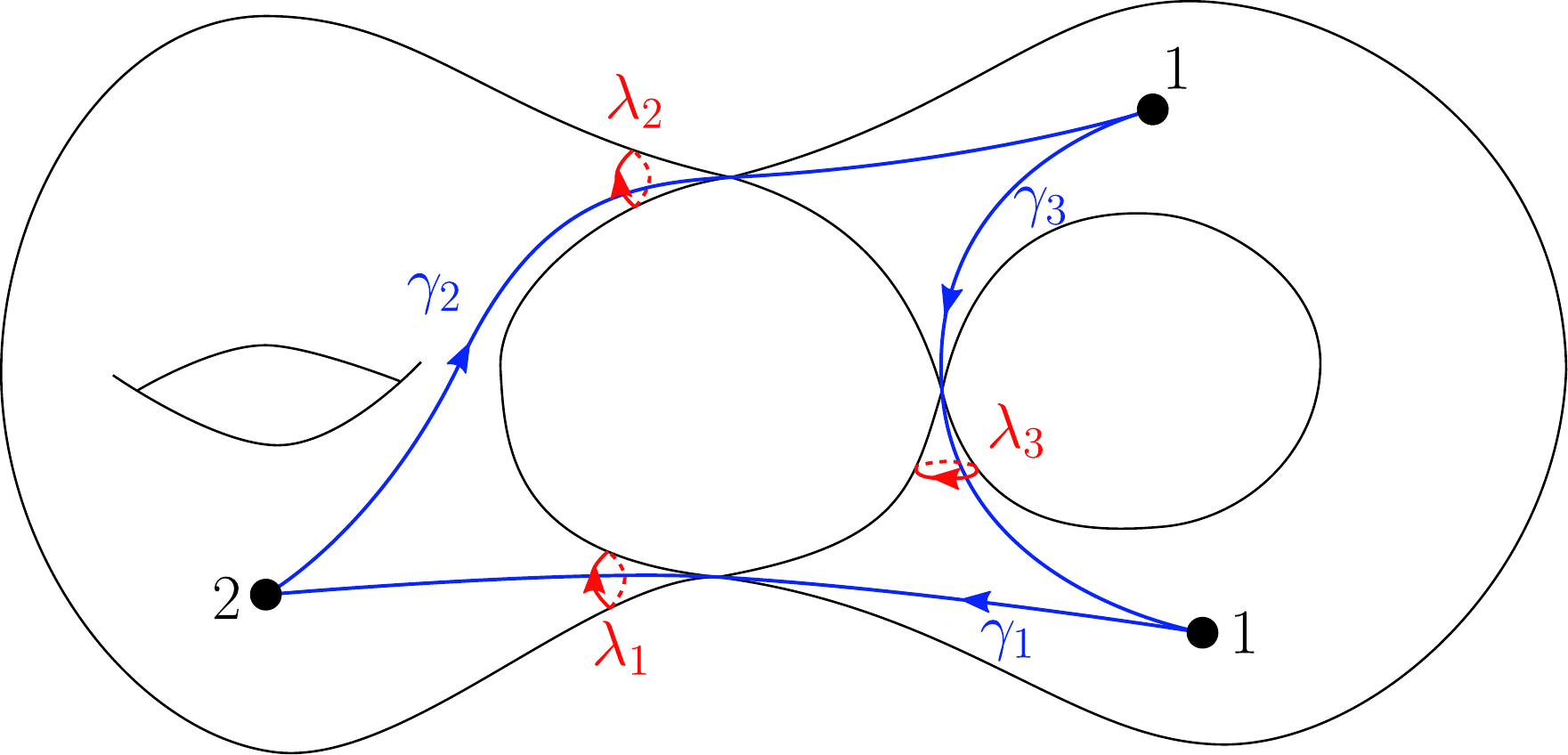}
  \end{minipage}

  \begin{minipage}[t]{0.35\linewidth}
    \vspace{0pt}
    \centering
    \renewcommand{\arraystretch}{1.7}
    \begin{tabular}{c}
      Period equations cutting out $M$ \\\hline
      $\int_{\gamma_1}\omega + \int_{\gamma_2} \omega+ \int_{\gamma_3}\omega = 0$ \\
      $\int_{\lambda_1}\omega + \int_{\lambda_2} \omega+ \int_{\lambda_3}\omega = 0$
    \end{tabular}
  \end{minipage}
  \caption{A local example in $\Omega\calM_{3,3}(1,1,2)$}
  \label{fig:NoPairwise}
\end{figure}

Note that \eqref{item:lin} follows from the way we've defined $M$.

To verify condition \eqref{item:an}, it suffices to convert the period coordinate equations into equations in analytic coordinates around $p_0$, using the ideas of \Cref{sec:transverse}. Indeed, the second equation, which involves the periods over horizontal vanishing cycles $\lambda_i$ already extends to an analytic equation.  The first equation, which involves the $\gamma_i$, can be written in terms of log periods as
\begin{align*}
  \left( \Psi_{\gamma_1} + \int_{\lambda_1} \omega \cdot \ln s_1\right) +   \left( \Psi_{\gamma_2} + \int_{\lambda_2} \omega \cdot \ln s_2\right) +   \left( \Psi_{\gamma_3} + \int_{\lambda_3} \omega \cdot \ln s_3\right) = 0.
\end{align*}
Since $\lambda_1$ and $\lambda_2$ are equal in absolutely homology, and using the second equation relating periods over $\lambda_i$, the above equation becomes
\begin{align*}
  \left( \Psi_{\gamma_1} + \Psi_{\gamma_2} + \Psi_{\gamma_3} \right) + \left(\int_{\lambda_1} \omega\right) (\ln s_1 + \ln s_2 - 2\ln s_3) = 0.
\end{align*}
Exponentiating both sides and rearranging gives
\begin{align*}
  \exp \left( \Psi_{\gamma_1} + \Psi_{\gamma_2} + \Psi_{\gamma_3} \right) s_1 s_2 -s_3^2 = 0
\end{align*}
This is an analytic equation in analytic coordinates around $p_0$, satisfied at the point $p_0$ where $s_1=s_2=s_3=0$ at $p_0$.  This establishes \eqref{item:an}.

For \eqref{item:symp}, observe that the two equations, corresponding to $\gamma_1+\gamma_2+\gamma_3$ and $\lambda_1, \lambda_2, \lambda_3$, are both absolute homology classes, and their intersection pairing is non-zero.

Having verified that $M$ is in fact a local example, we now show that \Cref{prop:pairwise-cross} does not hold in this setting.  Indeed, the first defining equation of~$M$ crosses all three horizontal vanishing cycles $\lambda_1,\lambda_2,\lambda_3$, and is the only defining equations of~$M$ that crosses any of these three horizontal vanishing cycles.  Hence $\lambda_1,\lambda_2$ are $M$-cross-related, but there is no defining equation~$F$  that crosses $\lambda_1,\lambda_2$ but not $\lambda_3$.
\end{exa}

\subsection{Boundary stratification of affine invariant manifolds}
For potential use for classifying affine invariant manifolds by recursively applying degeneration techniques, we record here two results on their boundaries. The first result applies more generally, for real-linear subvarieties of meromorphic strata --- the context in which our generalization of Cylinder Deformation Theorem applies.
\begin{prop}\label{prop:noncompact}
  Let~$M$ be a linear subvariety defined by equations with all coefficients real. Then for any $\Gamma$ such that $\dim\pMG>0$, the boundary stratum $\pMG\subset\oM$ is not compact.
\end{prop}
\begin{proof}
By~\cite{fred}, the boundary~$\pMG$ is a product of linear subvarieties of various strata of differentials, also all defined by linear equations with real coefficients. Suppose for contradiction that $\dim\pMG$ is compact. Then for any irreducible component $X_v$ of $X$, the image of $\pMG$ under the projection to the moduli of differentials on the corresponding component would have to be compact. Since $\dim\pMG>0$, there must exist some $X_v$ such that the image of $\pMG$ gives a positive-dimensional linear subvariety of the space of differentials on that component. If the differential considered on $X_v$ is holomorphic, we claim that this is impossible since there does not exist any compact $\GL^+(2,\RR)$ affine invariant manifold (because we can choose a saddle connection and make it shrink to zero under the action of $\operatorname{SL}(2,\RR)$). If the differential considered on $X_v$ is meromorphic, then by~\cite{chen_affine} the corresponding stratum of differentials does not contain any compact complex subvarieties. Thus in either case we have a contradiction.
\end{proof}
While we do not know a counterexample to this Proposition for linear subvarieties defined by equations with complex coefficients, note that the proof uses $\GL^+(2,\RR)$ action for holomorphic components, as it is not known whether holomorphic strata contain any compact complex subvarieties. Even if we start with a linear subvariety of a meromorphic stratum, it could be that in the boundary there is a top level component where the differential is holomorphic.

Applying this Proposition recursively, given a real-linear subvariety~$M$ of complex dimension $a$, one can consider a $\Gamma_1$ that corresponds to a divisorial boundary component, that is such that $\dim_\CC\pMG[\Gamma_1]=a-1$, and then since $\pMG$ is not compact, consider $\Gamma_2$ such that $\dim_\CC\pMG[\Gamma_2]=a-2$, and so on, thus constructing a sequence of divisorial degenerations of length precisely $a$:
$$
 pt=\Gamma_0\rightsquigarrow\dots\rightsquigarrow\Gamma_a;\quad\hbox{such that}\quad\dim_\CC\pMG[\Gamma_i]=a-i.
$$
Compare this to~\Cref{rem:chain}, where many sequences of divisorial degenerations are constructed, {\em starting from a given $\lG$}. Here we claim that choosing at each step any divisorial boundary component, we can always construct a sequence of degenerations of length~$a$, i.e. going down to a point. For the case of affine invariant manifolds, we can say a bit more.
\begin{cor}\label{cor:hashorizontal}
Let~$M$ be an affine invariant manifold, i.e.~a real-linear subvariety of a holomorphic stratum. If $\pMG$ is a deepest stratum of $\oM$, that is $\pMG[\Gamma']=\emptyset$ for any degeneration $\Gamma\rightsquigarrow\Gamma'$, then every top level vertex of $\Gamma$ has a horizontal edge attached to it.
\end{cor}
\begin{proof}
By the Proposition, being deepest is equivalent to $\pMG$ simply being a point. As in the proof of the Proposition, the projection of $\pMG$ to the stratum corresponding to some top level component $X_v$ must be a real-linear subvariety of that stratum, which in this case must be just one point. If $X_v$ has no horizontal nodes, then the twisted differential $\eta_v$ on $X_v$ has no poles: it has zeroes at any marked points $z_i$, and possibly zeroes at the vertical nodes. But then, again as in the proof above, there does not exist any flat surface in a holomorphic stratum fixed by $\GL^+(2,\RR)$, so we have a contradiction.
\end{proof}
We note that this Corollary is false for real-linear subvarieties of meromorphic strata. For example, for the case of the closure of the Hurwitz space considered in~\Cref{sec:HW}, the boundary points have no horizontal nodes whatsoever. This is also the case for the closure of the double ramification cycle considered in~\cite{fredDR}.

\bibliographystyle{amsalpha}
\bibliography{plumb_biblio}

\end{document}